\documentclass[11pt,  reqno]{amsart}

\usepackage{amsmath,amssymb,amscd,amsthm,amsxtra, esint}
\usepackage[implicit=true]{hyperref}
\usepackage[all]{xy}

\usepackage{color}
\usepackage{ulem}
\usepackage[makeroom]{cancel}

\allowdisplaybreaks[2]

\definecolor{gr}{rgb}   {0.,   0.69,   0.23 }
\definecolor{bl}{rgb}   {0.,   0.5,   1. }
\definecolor{mg}{rgb}   {0.85,  0.,    0.85}
\definecolor{yl}{rgb}   {0.8,  0.7,   0.}
\definecolor{or}{rgb}  {0.7,0.2,0.2}

\newtheorem{theorem}{Theorem} [section]

\newtheorem{lemma}[theorem]{Lemma}
\newtheorem{proposition}[theorem]{Proposition}
\newtheorem{remark}[theorem]{Remark}

\newtheorem{definition}[theorem]{Definition}
\newtheorem{corollary}[theorem]{Corollary}

\DeclareMathOperator*{\intt}{\int}

\newcommand{\I}{\hspace{0.5mm}\text{I}\hspace{0.5mm}}
\newcommand{\II}{\text{I \hspace{-2.8mm} I} }
\newcommand{\III}{\text{I \hspace{-2.9mm} I \hspace{-2.9mm} I}}

\newcommand{\noi}{\noindent}
\newcommand{\Z}{\mathbb{Z}}
\newcommand{\R}{\mathbb{R}}
\newcommand{\C}{\mathbb{C}}
\newcommand{\T}{\mathbb{T}}
\newcommand{\NB}{\mathbb{N}}

\newcommand{\RR}{\mathcal{R}}

\newcommand{\N}{\mathcal{N}}
\newcommand{\TT}{\mathcal{T}}
\newcommand{\F}{\mathcal{F}}
\newcommand{\FL}{\mathcal{F}L}

\newcommand{\al}{\alpha}
\newcommand{\be}{\beta}
\newcommand{\dl}{\delta}

\newcommand{\eps}{\varepsilon}

\newcommand{\g}{\gamma}
\newcommand{\G}{\Gamma}

\newcommand{\s}{\sigma}

\newcommand{\ft}{\widehat}

\newcommand{\wt}{\widetilde}
\newcommand{\cj}{\overline}
\newcommand{\dx}{\partial_x}
\newcommand{\dt}{\partial_t}

\newcommand{\embeds}{\hra}

\newcommand{\ta}{\theta}
\renewcommand{\l}{\ell}

\newcommand{\les}{\lesssim}

\newcommand{\jb}[1]
{\langle #1 \rangle}

\newcommand{\ind}{\mathbf 1}

\renewcommand{\S}{\mathcal{S}}

\newcommand{\M}{\mathcal{M}}

\newtheorem*{ackno}{Acknowledgement}

\numberwithin{equation}{section}
\numberwithin{theorem}{section}

\newcommand{\Ns}{\textsf{N}}
\newcommand{\Rs}{\textsf{R}}

\newcommand{\hra}{\hookrightarrow}

\newcommand{\uu}{{\bf u}}
\newcommand{\vv}{{\bf v}}
\newcommand{\n}{{\bf n}}

\begin{document}
\baselineskip = 15pt

\title[Normal form approach to NLS in Fourier-Lebesgue spaces]
{Normal form approach to the one-dimensional periodic cubic nonlinear Schr\"odinger equation in 
 almost critical Fourier-Lebesgue spaces}

\author[T.~Oh and  Y.~Wang]
{Tadahiro Oh and Yuzhao Wang}

\address{
Tadahiro Oh, School of Mathematics\\
The University of Edinburgh\\
and The Maxwell Institute for the Mathematical Sciences\\
James Clerk Maxwell Building\\
The King's Buildings\\
Peter Guthrie Tait Road\\
Edinburgh\\ 
EH9 3FD\\
 United Kingdom}

\email{hiro.oh@ed.ac.uk}

\address{
Yuzhao Wang\\
School of Mathematics\\
Watson Building\\
University of Birmingham\\
Edgbaston\\
Birmingham\\
B15 2TT\\ United Kingdom}

\email{y.wang.14@bham.ac.uk}

\subjclass[2010]{35Q55}

\keywords{nonlinear Schr\"odinger equation; normal form reduction;  unconditional uniqueness; Fourier Lebesgue space}

\begin{abstract}
In this paper, we study the one-dimensional cubic nonlinear Schr\"odinger equation (NLS)
on the circle. 
In particular, we develop a normal form approach to study NLS in almost critical Fourier-Lebesgue 
spaces.
By applying an infinite iteration of normal form reductions
introduced by the first author with Z.\,Guo and S.\,Kwon (2013), 
we derive a normal form equation
which is equivalent to the renormalized  cubic NLS for regular solutions.
For rough functions, the normal form equation behaves better
than the renormalized cubic NLS, thus providing a further renormalization
of the cubic NLS.
We then prove
that this normal form equation is unconditionally globally
well-posed in the Fourier-Lebesgue spaces $\F L^p(\T)$, $1 \leq p < \infty$.
By inverting the transformation, 
we  conclude global well-posedness of the renormalized cubic NLS
in almost critical Fourier-Lebesgue spaces in a suitable sense.
This approach also allows us to prove unconditional 
uniqueness of the (renormalized) cubic NLS in 
$\F L^p(\T)$ for $1\leq p \leq  \frac32$.

\end{abstract}

%
\maketitle
%

\baselineskip = 14pt

\section{Introduction}

\subsection{Nonlinear Schr\"odinger equation}

We consider 
the following cubic nonlinear Schr\"odinger equation (NLS)  on the circle $\T = \R/\Z$:
\begin{align}
\begin{cases}
i \dt u + \dx^2  u   \pm  | u |^{2} u = 0 \\
u |_{t = 0} = u_0,
\end{cases}
\quad (x, t) \in \T \times \R.
\label{NLS1}
\end{align}

\noi
The equation \eqref{NLS1} arises from various physical settings 
such as nonlinear optics and quantum physics.
See \cite{SS} for the references therein.
It is also known to be one of the simplest completely integrable PDEs \cite{ZS, AKNS,AM, GK, KVZ}.

The Cauchy problem \eqref{NLS1} has been studied extensively 
both on the real line and on the circle.
See \cite{OS, GO}  for the references therein.
In this paper, we study the periodic cubic NLS \eqref{NLS1}
in the Fourier-Lebesgue spaces
 $\FL^p(\T)$ defined via the norm:
\[
\| f \|_{\FL^p (\T)} : = \bigg( \sum_{n \in \Z} |\ft f (n)|^p \bigg)^{\frac1p}
\] 

\noi
with a usual modification when $p = \infty$.
For any $2 \le p \leq q \leq \infty$,
we have the following 
continuous embeddings:
\begin{align*}
\F L^1(\T) \hookrightarrow 
\F L^{q'}(\T)  \hookrightarrow 
\F L^{p'}(\T)
&  \hookrightarrow 
 \F L^2(\T) \\
&  = L^2(\T) \hookrightarrow 
\F L^p(\T)  \hookrightarrow 
\F L^q(\T)  \hookrightarrow 
\F L^\infty(\T).
\end{align*}

\noi
The space $\F L^1(\T)$ is the Wiener algebra.
The space  $\F L^\infty(\T)$ is the space of pseudo-measures, 
which contains all finite Borel measures on $\T$
but also more singular  distributions.  See \cite{Katz}.
Our main interest is to study \eqref{NLS1}
in $\FL^p(\T)$ for $p \gg 1$.

On the one hand, 
the cubic NLS \eqref{NLS1} is known to be 
globally well-posed in $\FL^2(\T) = L^2(\T) $ \cite{BO1}.
On the other hand, 
combining the known results \cite{GH, GO, OW}, 
we can easily show that 
it is ill-posed
in the Fourier-Lebesgue space
$\FL^p(\T)$ for $p > 2$ in a very strong sense.
See Proposition~\ref{PROP:NE} below.
This necessitates us to renormalize the nonlinearity
and consider the following renormalized cubic NLS:
\begin{align}
\begin{cases}
i \dt u + \dx^2  u   \pm  \big( | u |^{2} - 2 \int_\T \ |{u}|^2 dx\big)  u = 0 \\
u |_{t = 0} = u_0 .
\end{cases}
\label{WNLS1}
\end{align}

\noi
Note that the renormalized cubic NLS \eqref{WNLS1}
is ``equivalent'' to the original cubic NLS \eqref{NLS1}
for smooth solutions in the following sense.
For  $u \in C(\R; L^2(\T))$, 
we define the following invertible gauge transformation $\mathcal{G}$ by 
\begin{equation*}
\mathcal{G}(u)(t) : = e^{ \mp 2 i t \int_\T |u(t)|^2 dx} u(t)
\end{equation*}

\noi
with its inverse 
\begin{equation}
\mathcal{G}^{-1}(u)(t) : = e^{\pm 2 i t \int_\T |u(t)|^2 dx } u(t).
\label{gauge2}
\end{equation}

\noi
Then, thanks to the $L^2$-conservation, it is easy to see that 
 $u \in C(\R; L^2(\T))$ is a solution to \eqref{NLS1}
 if and only if  $\mathcal G(u)$
is a solution to the renormalized cubic NLS \eqref{WNLS1}.
This renormalization removes a certain singular component
from the nonlinearity and, as a result, 
the renormalized cubic NLS \eqref{WNLS1}  
 behaves better than  the cubic NLS~\eqref{NLS1} outside $L^2(\T)$.
The study of \eqref{WNLS1} outside $L^2(\T)$ has attracted much attention in recent years \cite{CH1, CH2, GH,OS,CO,   GO, OW, O, OW2}.

In \cite{GH}, Gr\"unrock-Herr adapted the Fourier restriction norm method
to the Fourier-Lebesgue space setting  and proved local well-posedness
of the renormalized cubic NLS~\eqref{WNLS1} in $\FL^p(\T)$ for $1\leq p < \infty$
by a standard contraction argument.
See also the work by Christ~\cite{CH2}.
In \cite{OW2},  by using the completely integrable structure of the equation, 
we established the following global-in-time a priori bound:
\begin{align}
\sup_{t \in \R} \| u(t) \|_{\FL^p} \leq C(\| u_0\|_{\FL^p})
\label{bound1}
\end{align}

\noi
for any smooth solution $u$ to the renormalized cubic NLS \eqref{WNLS1}
and $2 \leq p < \infty$, 
which implied  global well-posedness
 of \eqref{WNLS1} in $\F L^p(\T)$ for $1\leq p < \infty$.\footnote{For $1\leq p < 2$, 
 one needs to use the $L^2$-conservation and a persistence-of-regularity argument.
 See Appendix \ref{SEC:A}.}

As a corollary to the local well-posedness
of the renormalized cubic NLS in \cite{GH}, 
one easily  obtains the following non-existence result
for the original cubic NLS \eqref{NLS1} outside $L^2(\T)$.

\begin{proposition}
\label{PROP:NE}

Let $2 < p < \infty$
and $u_0 \in \FL^p(\T)\setminus L^2(\T)$.
Then, for any $T>0$, 
there exists no distributional  solution $u \in C([-T, T]; \FL^p(\T))$ to 
the cubic  NLS \eqref{NLS1}
such that

\begin{itemize}
\item[\textup{(i)}] $u|_{t = 0} = u_0$, 

\smallskip

\item[\textup{(ii)}] There exist smooth global solutions $\{u_n\}_{n\in \NB}$ 
to \eqref{NLS1} such that 
$u_n \to u$ in $ C([-T, T]; \mathcal{D}'(\T))$ as $n \to \infty$. 
\end{itemize}

\end{proposition}

In \cite{GO}, the first author (with Z.\,Guo) proved
an analogous non-existence result
for \eqref{NLS1} in negative Sobolev spaces.
The argument was based on an a priori bound
for smooth solutions to the renormalized cubic NLS \eqref{WNLS1}
in negative Sobolev spaces
and exploiting a fast oscillation in \eqref{gauge2}.
The proof of the local well-posedness in \cite{GH}
yields  an a priori bound
for smooth solutions to the renormalized cubic NLS \eqref{WNLS1}
in $\FL^p(\T)$.
Then, we can prove Proposition \ref{PROP:NE}
by proceeding as in  \cite{GO, OW3}.
We omit details.

In the following, we only consider the focusing case
(i.e. with the $+$ sign in \eqref{NLS1} and~\eqref{WNLS1})
for simplicity.
Our main results equally apply to the defocusing case.

\subsection{Main results}

In the following, we introduce two notions of weak solutions.
Let $\N(u)$ denote the renormalized nonlinearity in \eqref{WNLS1}:\footnote{Hereafter, 
we drop the factor of $2\pi$ when it plays no role.}
\begin{align}
\begin{split}
\N(u) 
: &  = \bigg( | u |^{2} - 2 \int_\T \ |{u}|^2 dx\bigg)  u \\
& = \sum_{n_2 \ne n_1, n_3} \ft{{u}}(n_1)\cj{\ft u(n_2)}\ft{{u}}(n_3) 
	e^{i(n_1 - n_2 + n_3)x} - 
	\sum_{n\in \Z} |\ft{{u}}(n)|^2\ft{{u}}(n) e^{inx}.
\end{split}
\label{non1}
\end{align}

\noi
We first recall the following notion
of {\it weak solutions in the extended sense}.

\begin{definition}
\label{DEF:sol2}
\rm
Let $1\leq p < \infty$ and $T>0$. 

\smallskip

\noi
(i) We define a sequence of Fourier cutoff operators to be a sequence of Fourier multiplier operators $\{T_N\}_{N\in \NB}$
on $\mathcal{D'}(\T)$ with multipliers $m_N:\Z \to \mathbb{C}$ such that 
\begin{itemize}
\item
 $m_N$ has a compact support on $\Z$ for each $N \in \NB$, 
 \item  $m_N$ is uniformly bounded, 

\item 
 $m_N$ converges pointwise to $1$, i.e. $\lim_{N\to \infty} m_N(n) = 1$ for any $n \in \Z$.
\end{itemize}

\smallskip

\noi
(ii) 
Let $u \in C([-T, T]; \FL^p(\T))$. 
We say that $\N(u)$ exists and is equal to a distribution 
$v \in \mathcal{D}'(\T\times (-T, T))$
if for every sequence $\{T_N\}_{N\in \NB}$ of (spatial) Fourier cutoff operators, we have
\begin{equation*}
\lim_{N\to \infty} \N(T_N u) = v
\end{equation*}

\noi
in the sense of distributions on $\T\times (-T, T)$.

\smallskip

\noi
(iii) (weak solutions in the extended sense)
We say that $u  \in C([-T, T]; \FL^p (\T))$  is a weak solution of 
the renormalized cubic NLS \eqref{WNLS1}
in the extended sense
if 
\begin{itemize}
\item $u|_{t=0} = u_0$, 

\item the nonlinearity $\N(u)$ exists in the sense of (ii) above, 

\item  $u$ satisfies \eqref{WNLS1}
in the distributional sense  on $\T\times (-T, T)$,
where the nonlinearity $\N(u) $ is interpreted as above. 

\end{itemize}

\end{definition}

In  \cite{CH1, CH2}, Christ introduced this notion
in studying the renormalized cubic NLS \eqref{WNLS1}
in the low regularity setting.
See also \cite{Gub} for a similar notion of weak solutions,
where the nonlinearity is defined as a distributional limit
of smoothed nonlinearities.

Next, we introduce the following notion of {\it sensible weak solutions}.
See also \cite{OW2, FO}.

\begin{definition}[sensible weak solutions]\label{DEF:sol} \rm
Let $1\leq p < \infty$ and $T>0$. 
Given $u_0 \in \FL^p(\T)$, 
we say that
 $u \in C([-T,T]; \FL^p(\T))$
is a sensible weak solution
to the renormalized cubic NLS \eqref{WNLS1}  on $[-T, T]$ if,
for any sequence $\{u_{0, m}\}_{m \in \NB}$ of smooth functions
tending to $u_0$ in $\FL^p(\T)$, 
the corresponding (classical) solutions $u_m$ with $u_m|_{t = 0} = u_{0, m}$
converge to $u$ 
in $C([-T,T]; \F L^p (\T))$.
Moreover, we impose that there exists  a distribution $v$
such that $\N (u_m)$ converges to  $v$ in the space-time distributional sense,
independent of the choice of the approximating sequence.

\end{definition}

Note that, by using the equation, 
the convergence of  $u_m$ to $u$ 
in $C([-T,T]; \F L^p (\T))$
implies that 
$\N (u_m)$ converges to some  $v$ in the space-time distributional sense;
see \eqref{diff3} below.
Hence, the last part of Definition \ref{DEF:sol}
is not quite necessary.
We, however, keep it for clarity.

We point out that these notions of  weak solutions 
in Definitions \ref{DEF:sol2} and \ref{DEF:sol}
are  rather weak.
The cubic nonlinearity $\N(u)$ for a  weak solution $u$
in the sense of  Definitions \ref{DEF:sol2} or~\ref{DEF:sol}
does not directly make sense as a distribution in general
and we need to interpret it as a (unique) limit of smoothed nonlinearities $\N(T_N u)$
or the nonlinearities $\N(u_m)$ of smooth approximating solutions $u_m$.
This in particular
implies that 
weak solutions 
in the sense of  Definitions \ref{DEF:sol2} or \ref{DEF:sol}
 do not have to satisfy 
the equation even in the distributional sense.

On the one hand, sensible weak solutions are unique by definition.
On the other hand, weak solutions in the extended sense
are not  unique in general.
In fact, Christ \cite{CH1}
proved non-uniqueness of weak solutions in the extended sense 
for the renormalized cubic NLS~\eqref{WNLS1}
in negative Sobolev spaces.

\medskip

Our main goal in this paper is
(i) to develop further the normal form approach to study 
the (renormalized) cubic NLS, 
introduced in \cite{GKO}, 
and provide the solution theory
for \eqref{WNLS1} 
in almost critical Fourier-Lebesgue spaces
(Theorem~\ref{THM:1})
in the sense of Definitions~\ref{DEF:sol2} and~\ref{DEF:sol}
{\it without} using  any auxiliary function spaces, 
 in particular, {\it without}
using the Fourier restriction norm method
as in \cite{BO1, GH}
and 
(ii) to prove 
unconditional uniqueness
of the (renormalized) cubic NLS
in 
 $\FL^p(\T)$ for $1 \leq p < \frac{3}{2}$
 (Theorem~\ref{THM:2}).
In proving these results, 
we  apply an 
 infinite iteration of normal form reductions
 and
transform
the (renormalized) cubic NLS
into the so-called {\it normal form equation}.
We then prove 
 unconditional well-posedness of 
 the normal form equation 
 in 
 $\FL^{p}(\T)$ for {\it any} $1 \leq p < \infty$;
see  Theorem~\ref{THM:3} below.
 
We now state our main results.

\begin{theorem}\label{THM:1}
Let $1 \leq p < \infty$.
Then, the renormalized cubic NLS \eqref{WNLS1}
on $\T$ is globally well-posed in $\FL^p(\T)$
\begin{itemize}
\item
in the sense of weak solutions in the extended sense and

\item in the sense of sensible weak solutions.

\end{itemize}

\noi
When $1 \leq p \leq 2$, 
the same global well-posedness result applies to the \textup{(}unrenormalized\textup{)} cubic NLS \eqref{NLS1}.

\end{theorem}

This theorem follows from the local well-posedness
by Gr\"unrock-Herr \cite{GH}, 
combined with the a priori bound \eqref{bound1}
from \cite{OW2}.
As pointed out above, however, 
our main goal is to present an argument 
 independent of  the Fourier restriction norm method.
We instead employ the normal form approach
developed in \cite{GKO}.
Our approach does not involve any auxiliary function spaces
and consequently  allows us to prove unconditional uniqueness 
of the (renormalized) cubic NLS 
in $\F L^\frac{3}{2}(\T)$ (Theorem \ref{THM:2}).
We point out that the local well-posedness
in \cite{GH} only yields 
 conditional uniqueness, namely in the class \eqref{class1} below.

In \cite{GKO}, 
the first author (with Z.\,Guo and S.\,Kwon)
proved an analogous result in $L^2(\T)$
by implementing an infinite iteration of normal form reductions,\footnote{In \cite{GKO}, 
we only proved well-posedness of the cubic NLS \eqref{NLS1}
in the sense of weak solutions in the extended sense.
A small modification of the argument yields
well-posedness in the sense of sensible weak solutions.
See Section \ref{SEC:2}.}
yielding unconditional uniqueness 
of the cubic NLS \eqref{NLS1} in $H^\frac{1}{6}(\T)$.
The proof of Theorem~\ref{THM:1}
is also based on the same normal form approach.
See the next subsection.
Note that when $p$ is very large, Theorem~\ref{THM:1}
is  significantly harder to prove than the $L^2$-result in \cite{GKO} due to a much weaker  $\FL^p$-topology.

Given $u_0 \in \FL^p(\T)$, let $u$ be the global solution to \eqref{WNLS1}
with $u|_{t = 0} = u_0$
constructed in Theorem \ref{THM:1}.
Then, by the uniqueness of sensible solutions mentioned above, 
$u$ must coincide with the global solution constructed
in \cite{BO1, GH, OW2}.
In particular, 
the solution $u$ belongs to the class
\begin{align}
C([-T,T]; \FL^p (\T)) \cap X^{0,b}_{p} ([-T, T])
\label{class1}
\end{align}

\noi
for some $b > \frac{1}{p'}$, 
where $X^{0,b}_p([-T, T])$ denotes the local-in-time version of 
the Fourier restriction space $X^{0,b}_{p}$ 
adapted to the Fourier-Lebesgue setting.
See \eqref{Xsb1} and \eqref{Xsb3} below.

As mentioned above, 
 Theorem \ref{THM:1}
 does not allow us to directly\footnote{That is, 
 unless we use the uniqueness property of sensible solutions
 and conclude that they belong to the class \eqref{class1}
 by comparing with the solutions constructed in \cite{BO1, GH, OW2}.}
  conclude that 
 weak solutions constructed in Theorem \ref{THM:1}
  are distributional solutions to \eqref{WNLS1}.
For $1\leq p \leq \frac 32$, however, 
Hausdorff-Young's inequality: $\FL^p(\T) \subset \FL^\frac{3}{2}(\T) \subset L^3(\T)$
allows us to make sense of the cubic nonlinearity
in a direct manner.
In this case, we have the following  
uniqueness statement.

\begin{theorem}\label{THM:2}

Let $1\leq p \leq \frac{3}{2}$.
Then, given any   $u_0 \in \FL^p(\T)$, 
the solution $u $ to \eqref{NLS1} or~\eqref{WNLS1} with $u|_{t = 0} =  u_0$ 
constructed in Theorem \ref{THM:1}
is unique in $C(\R; \FL^p(\T))$.

\end{theorem}

Namely, unconditional uniqueness holds
for both the cubic NLS \eqref{NLS1}
and the renormalized cubic NLS \eqref{WNLS1}
in $\FL^p(\T)$, provided that $1\leq p \leq \frac{3}{2}$.
In \cite{GKO}, 
the first author (with Z.\,Guo and S.\,Kwon)
proved unconditional uniqueness in $H^\frac{1}{6}(\T)$
and Theorem \ref{THM:2} extends this result to the Fourier-Lebesgue setting. 
We also mention a recent work by Herr-Sohinger \cite{HS}
where they proved unconditional uniqueness of the cubic NLS \eqref{NLS1} 
in $L^p([-T, T] \times \T)$ for $p > 3$.
The main difference between unconditional uniqueness
and uniqueness for sensible weak solutions
is that the former does not assume that a solution comes with a sequence of smooth approximating
 solutions, while, by definition,  sensible weak solutions are equipped
with smooth approximating solutions.

\begin{remark}\rm

When $p = \infty$, 
the Fourier-Lebesgue space  $\FL^\infty(\T)$ does not admit smooth approximations
and hence is not suitable for well-posedness study.
Given $s \in \R$ and $1\leq p \leq \infty$, 
define $\FL^{s, p} (\T)$ by the norm:
\begin{align}
 \| f \|_{\FL^{s, p}} : = \|\jb{n}^s \ft f (n)\|_{\l^p_n(\Z)}.
 \label{FL1}
\end{align}

\noi
Note that $\FL^p(\T) = \FL^{0, p}(\T)$.
For $s < -\frac1p$,  we have $ \FL^\infty(\T) \subset \FL^{s, p} (\T)$
and thus we may wish to  study well-posedness
in $\FL^{s, p} (\T)$ for finite $p$ with  $s < -\frac1p$
since this space admits smooth approximations.
On the other hand, 
 the scaling critical regularity for the cubic NLS \eqref{NLS1} 
 with respect to the Fourier-Lebesgue spaces $\FL^{s, p}(\T)$
 is given by 
$s_\text{crit} = -\frac{1}{p}$.
In particular, 
 the cubic NLS \eqref{NLS1}
and the renormalized cubic NLS \eqref{WNLS1}
are known to be ill-posed in the (super)critical regime.\footnote{In fact, 
it is shown in \cite{Kis} that the cubic NLS \eqref{NLS1}
and the renormalized cubic NLS \eqref{WNLS1}
are  ill-posed even in the logarithmically subcritical regime.}
When $s < 0$, it is easy to modify the argument in 
\cite{BGT, CCT1, CO}
and show that the solution map is not locally uniformly continuous in 
$\FL^{s, p} (\T)$.
Furthermore,  when $s \leq s_\text{crit} = -\frac 1p$, 
the cubic NLS \eqref{NLS1}
and the renormalized cubic NLS~\eqref{WNLS1}
admit norm inflation;
given any $\eps > 0$, 
there exist a solution $u$ to \eqref{NLS1} or~\eqref{WNLS1} 
and $t  \in (0, \eps) $ such that 
\begin{align*}
 \| u(0)\|_{\FL^{s, p}} < \eps \qquad \text{ and } \qquad \| u(t)\|_{\FL^{s, p}} > \eps^{-1}.
 \end{align*} 

\noi
See \cite{Kis}.
The norm inflation in particular implies discontinuity of the solution map at the trivial function\footnote{One can easily combine
the argument in \cite{Kis, O} to prove norm inflation at general initial data, concluding discontinuity
of the solution map at every function $\FL^{s, p}(\T)$, provided that $s \leq s_\text{crit} = - \frac 1p$.} $u \equiv0$.
Lastly, 
a typical function in  $\FL^\infty(\T)$ is the Dirac delta function
and 
\eqref{NLS1} and~\eqref{WNLS1} on $\T$ are known to be ill-posed with the Dirac delta function as initial data;
see \cite{FO}.
See also  Kenig-Ponce-Vega \cite{KPV}
and Banica-Vega \cite{BV1, BV2}  for the works on the cubic NLS \eqref{NLS1}
on the real line
with the Dirac delta function
as initial data.

\end{remark}

\begin{remark}\rm
Following the argument in \cite{GKO}, 
we can easily extend Theorem \ref{THM:1} to  $\FL^{s, p}(\T)$
for $s > 0$ and $1\leq p < \infty$.
Similarly, the unconditional uniqueness result in Theorem \ref{THM:2}
can be extended to $\FL^{s, p}(\T)$
for (i) $s > 0$ and $1\leq p \leq \frac{3}{2}$
and (ii) $p > \frac 32$ and $s > \frac{2p-3}{3p}$.
Note that in these ranges of $(s, p)$, we have $\FL^{s, p}(\T)\hookrightarrow \FL^\frac{3}{2}(\T)
\hookrightarrow L^3(\T)$.

\end{remark}

\subsection{Normal form equation}\label{SUBSEC:1.3}

The main idea for proving Theorems \ref{THM:1} and \ref{THM:2}
is to apply an infinite iteration of normal form reductions
to \eqref{WNLS1}\footnote{In the following, we restrict our attention
to the renormalized cubic NLS \eqref{WNLS1}.
See Subsection \ref{SUBSEC:NLS}
for required modifications to handle the cubic NLS \eqref{NLS1}
in Theorem \ref{THM:2}.}
and transform the equation 
into a normal form equation (see \eqref{NFE1} below), 
which may look more complicated from the algebraic viewpoint
but exhibits better analytical properties than 
the original equation.

Let $S(t) = e^{i t\dx^2}$ denote the linear Schr\"odinger propagator.
We introduce the interaction representation:
\begin{align}
\uu(t) = S(-t) u(t) =  e^{-it\dx^2} {u}(t).
\label{IR1}
\end{align}

\noi
On the Fourier side, we have
 $\ft{\uu}(n, t) = e^{in^2t} \ft{u}(n,t)$. 
Then, \eqref{WNLS1} can be written as\footnote{Due to the presence of 
the time-dependent phase factor $e^{ i \Phi(\bar n) t}$, 
the non-resonant part $\N_1(\uu)$, viewed as a trilinear operator
is non-autonomous.
For notational simplicity, however, we suppress such $t$-dependence
when there is no confusion.
We apply this convention to  all the multilinear operators
appearing in this paper.
} 
\begin{align}
\begin{split}
\dt \ft \uu_n 
& =  i 
\sum_{\substack{n = n_1 - n_2 + n_3\\ n_2\ne n_1, n_3} }
e^{ i \Phi(\bar{n})t } 
\ft \uu(n_1) \cj{\ft \uu(n_2)}\ft \uu(n_3)
- i|\ft \uu(n)|^2 \ft \uu(n)  \\
& =:   \N_1(\uu)(n) +   \RR(\uu)(n) .
\end{split}
 \label{NLS4}
 \end{align}

\noi
Here,  the  phase function $\Phi(\bar{n})$ is defined by 
\begin{align}\label{Phi}
\Phi(\bar{n}):& = \Phi(n, n_1, n_2, n_3) = n^2 - n_1^2 + n_2^2- n_3^2 \notag \\
& = 2(n_2 - n_1) (n_2 - n_3)
= 2(n - n_1) (n - n_3),
\end{align}

\noi
where the last two equalities hold under $n = n_1 - n_2 + n_3$.
From \eqref{Phi}, we see that 
$\N_1$ corresponds to the non-resonant part (i.e. $\Phi(\bar{n})\ne 0$) of the nonlinearity
and $\RR$  corresponds to the resonant part.
Note that the Duhamel formulation
for \eqref{WNLS1}:
\begin{align*}
u(t) = S(t) u_0 + i \int_0^t S(t - t') \N(u)(t') dt'
\end{align*}	

\noi
is now expressed as a system of  integral equations:
\begin{align}
\ft \uu(n, t) = \ft u_0(n) +  \int_0^t \Big\{ \N_1(\uu)(n) +   \RR(\uu)(n)\Big\}(t')dt'
\label{DN2}
\end{align}

\noi
for $n \in \Z$.
In the following, the space $\FL^\frac{3}{2}(\T)$ plays an important role
and thus we introduce the following definition of {\it regular} solutions.

\begin{definition}\label{DEF:regular}\rm
We say that $u$ and $\uu$
are regular solutions
to \eqref{WNLS1} and \eqref{NLS4}, respectively, 
if $u$ and $\uu$ are solutions to 
to \eqref{WNLS1} and \eqref{NLS4}, respectively, 
such that 
$u \in C(\R; \FL^\frac{3}{2}(\T))$
and $\uu \in C(\R; \FL^\frac{3}{2}(\T))$, respectively.

\end{definition}

The main idea is to apply a normal form reduction to \eqref{NLS4},
namely integration by parts in~\eqref{DN2},
to exploit the oscillatory nature of the non-resonant contribution.
As in \cite{GKO, KOY}, 
we implement an infinite iteration of normal form reductions
and derive the following {\it normal form equation}:
\begin{align}
\begin{split}
\uu(t) 
 =   \uu(0) 
&  +    \sum_{j = 2}^\infty  \N_0^{(j)}(\uu)(t)
 - \sum_{j = 2}^\infty \N_0^{(j)}(\uu)(0)  \\
& 
+ \int_0^t \bigg\{
\sum_{j = 1}^\infty \N_{1}^{(j)}(\uu)(t')   + \sum_{j = 1}^\infty \RR^{(j)}(\uu)(t')\bigg\} dt',
\end{split}
\label{NFE1}
\end{align}

\noi
where $\{\N_0^{(j)}\}_{j=2}^{\infty}$ are time-dependent $(2j-1)$-linear operators
while  $\{\N_{1}^{(j)}\}_{j=1}^{\infty}$
and $\{\RR^{(j)}\}_{j=1}^{\infty}$
 are  time-dependent $(2j+1)$-linear operators.
As we see in Section \ref{SEC:3}, 
multilinear dispersion effects are already embedded
in  these multilinear terms, 
which allows us to
prove that these multilinear operators
are bounded
 in $C ([-T,T]; \FL^p(\T))$ for any $1\leq p < \infty$.
 Moreover, we show that the normal form equation~\eqref{NFE1}
 is equivalent to \eqref{NLS4} and the renormalized cubic NLS~\eqref{WNLS1}
 in $C(\R; \FL^\frac{3}{2}(\T))$.
 See Proposition~\ref{PROP:bound}.
As a consequence, 
we can easily prove local well-posedness
of the normal form equation \eqref{NFE1}
in $\FL^p(\T)$ by a simple contraction argument
{\it without} any auxiliary function spaces.

\begin{theorem}\label{THM:3}
Let $1\leq p < \infty$.
Then, the normal form equation \eqref{NFE1}
is unconditionally globally well-posed in $\FL^p(\T)$.
\end{theorem}

In \cite{GKO}, an analogous result was shown
in $L^2(\T)$.
When $p > 2$, 
the $\FL^p$-norm is weaker than the $L^2$-norm.
In particular, when $p \gg 1$, 
this fact makes it much harder to 
show  convergence of the series
in the normal form equation \eqref{NFE1}
with respect to the  $\FL^p$-topology. 

Once we establish the relevant multilinear estimates
(Proposition \ref{PROP:bound}), 
the proof of unconditional local well-posedness
for  the normal form equation \eqref{NFE1}
follows from a simple contraction  argument.
Moreover,  we show that 
the local existence time $T$ depends only on the size of the initial data $\| u_0\|_{\F L^p}$
and consequently, we conclude that 
solutions exist globally in time in view of the global-in-time bound \eqref{bound1}
from \cite{OW2}.  See also Appendix \ref{SEC:A}.

Finally,  note that Theorem \ref{THM:2}
follows easily thanks to the equivalence of \eqref{WNLS1} and the normal form equation \eqref{NFE1}
for regular solutions belonging to 
$C(\R; \FL^\frac{3}{2}(\T))$.
The contraction argument in proving Theorem \ref{THM:3}
yields the following Lipschitz bound:
\begin{align}
\sup_{t\in [-T, T]}\| \uu (t) - \vv(t) \|_{\FL^p}
\leq C(T, R) \| \uu (0) - \vv(0) \|_{\FL^p}
\label{DN3}
\end{align}

\noi
for any $T > 0$, where $R >0$ satisfies $\|\uu(0)\|_{\FL^p}, \|\vv(0)\|_{\FL^p} \leq R$.
Furthermore, from \eqref{WNLS1}, \eqref{NLS4}, and \eqref{NFE1} with \eqref{non1} and \eqref{IR1}, 
we obtain 
\begin{align}
\begin{split}
\int_0^t \N(u)(t') dt'  
&    =  S(t) \Bigg\{ \sum_{j = 2}^\infty  \N_0^{(j)}(S(-\,\cdot)u)(t)
 - \sum_{j = 2}^\infty \N_0^{(j)}(u)(0)  \\
& \hphantom{X}
+ \int_0^t \bigg[
\sum_{j = 1}^\infty \N_{1}^{(j)}(S(-\,\cdot)u)(t')   + \sum_{j = 1}^\infty \RR^{(j)}(S(-\,\cdot)u)(t')\bigg] dt'\Bigg\}.
\end{split}
\label{NFE2}
\end{align}

\noi
Then,  \eqref{DN3} and \eqref{NFE2}
together with the multilinearity of the summands in \eqref{NFE2}
and the unitarity of the linear operator $S(t)$ in $\FL^p(\T)$ 
allow us to conclude 
convergence of 
 smoothed nonlinearities $\N(T_N u)$
or the nonlinearities $\N(u_m)$ of smooth approximating solutions $u_m$
required in Definitions \ref{DEF:sol2} and \ref{DEF:sol}.
This is a sketch of the proof of 
Theorem~\ref{THM:1}.

In Section~\ref{SEC:2}, we present the proofs of the main results, 
assuming the bounds on the multilinear operators 
 $\{\N_0^{(j)}\}_{j=2}^{\infty}$,    $\{\N_{1}^{(j)}\}_{j=1}^{\infty}$, 
and $\{\RR^{(j)}\}_{j=1}^{\infty}$ (Proposition \ref{PROP:bound}).
In Section~\ref{SEC:3}, 
we implement an infinite iteration of normal form reductions
as in \cite{GKO} and prove Proposition~\ref{PROP:bound}.

\begin{remark}\rm
Let $ p > \frac 32$.
Given $u \in  C([-T, T]; \FL^p(\T))$, 
we can not, in general, make sense of the cubic nonlinearity $\N(u)$ as a distribution
since $\FL^p(\T) \not\subset L^3(\T)$.
In other words, 
we can not estimate the cubic nonlinearity  without
relying on some auxiliary function space.
In \eqref{NFE2}, we re-expressed the cubic nonlinearity
into series of the multilinear terms
of increasing degrees.
On the one hand, this transformation brings
algebraic complexity.
On the other hand,  
 the right-hand side of \eqref{NFE2} 
 is convergent
for  $u \in  C([-T, T]; \FL^p(\T))$, 
allowing us 
to make sense of the right-hand side of \eqref{NFE2}
as a distribution.
Namely, while the left-hand side of~\eqref{NFE2}
and the right-hand side of \eqref{NFE2}
coincide
for regular solutions $u \in  C([-T, T]; \FL^\frac{3}{2}(\T))$, 
 the right-hand side of \eqref{NFE2}
 provides a better formulation 
 of the nonlinearity
 for rougher functions $u \in C([-T, T]; \FL^p(\T))$, $\frac 32 < p < \infty$.
 In this sense, 
 we can view
 the right-hand side of \eqref{NFE2}
 as a further renormalization of the renormalized nonlinearity $\N(u)$
 in \eqref{non1}.

By expressing the  normal form equation \eqref{NFE1}
in terms of the original function  $u(t) = S(t)\uu(t)$, 
we obtain
 \begin{align}
\begin{split}
u(t) 
 =   S(t) u(0) 
&  +    S(t) \sum_{j = 2}^\infty  \Ns_0^{(j)}(u)(t) 
 -  S(t)\sum_{j = 2}^\infty \Ns_0^{(j)}(u)(0)  \\
& 
+ \int_0^t S(t - t ') \bigg\{
\sum_{j = 1}^\infty \Ns_1^{(j)}(u)(t')   + \sum_{j = 1}^\infty \Rs^{(j)}(u)(t')\bigg\} dt',
\end{split}
\label{NFE1a}
\end{align}

\noi
where 
\begin{align}
\begin{split}
\Ns^{(j)}_0(u)(t) & =    \N^{(j)}_0(S(-\,\cdot)u)(t),\\
\Ns^{(j)}_1(u)(t) & =  S(t) \N^{(j)}_1(S(-\,\cdot)u)(t),\\ 
\Rs^{(j)}(u)(t) & =  S(t)  \RR^{(j)}(S(-\,\cdot)u)(t).
\end{split}
\label{NFE1b}
\end{align}

\noi
As we see in Section \ref{SEC:3}, 
 the multilinear operators $S(t)\Ns_0^{(j)}(t)$, $\Ns_1^{(j)}$, and $\Rs^{(j)}$
are autonomous.  
The discussion above shows that 
 the normal form equation \eqref{NFE1a}
 expressed in terms of $u(t) = S(t)\uu(t)$
is a better model to study than the renormalized cubic NLS \eqref{WNLS1}
(and the cubic NLS \eqref{NLS1})
in the low regularity setting, 
which can be viewed as a further renormalization 
to the (renormalized) cubic NLS.

Lastly, we point out that the terms on the left-hand side of \eqref{NFE1b}
are indeed autonomous (unlike the non-autonomous multilinear terms in \eqref{NFE2}).
See Section \ref{SEC:3}.

\end{remark}

\begin{remark}\rm
A precursor to this normal form approach  first appeared in the work of
Babin-Ilyin-Titi \cite{BIT} in the study of  KdV on $\T$, establishing unconditional well-posedness of the KdV in $L^2(\T)$.
See also \cite{KO}.
In \cite{GKO}, the first author with Z.\,Guo and S.\,Kwon further developed this normal form approach
and introduced an infinite iteration scheme of normal form reductions
in the context of the cubic NLS on the circle.
In this series of work, 
 the viewpoint of unconditional well-posedness
 was first introduced in \cite{KO}, 
while the viewpoint of the (Poincar\'e-Dulac) normal form reductions
was first introduced in \cite{GKO}.
This normal form approach has also been used
to prove nonlinear smoothing \cite{ET},  improved energy estimates \cite{OST, OW3}, 
and construct 
 an infinite sequence of 
invariant quantities under the dynamics \cite{CGKO}.

\end{remark}

\begin{remark}\rm
In a recent paper \cite{FO}, the first author with Forlano
studied the cubic NLS on $\R$.
In particular, by implemented an infinite iteration
of normal form reductions, 
they proved analogues of Theorems \ref{THM:1}, \ref{THM:2},
and \ref{THM:3} in almost critical Fourier-Lebesgue spaces $\FL^p(\R)$, $2\leq p < \infty$, 
and almost critical modulation spaces $M^{2, p}(\R)$, $2\leq p < \infty$.
Relevant multilinear estimates were studied
based on the idea introduced in \cite{KOY}, 
namely, successive applications
of basic trilinear estimates (called  localized modulation estimates).

\end{remark}

\section{Proof of the main results}
\label{SEC:2}

In this section, we present the proofs of the main results 
(Theorems \ref{THM:1}, \ref{THM:2}, and \ref{THM:3}),
assuming the validity of 
the transformation of the equation \eqref{DN2}
to the normal form equation \eqref{NFE1}
and 
the boundedness of the multilinear operators in \eqref{NFE1} (Proposition~\ref{PROP:bound}).

\subsection{Series expansion of regular solutions}

In Section \ref{SEC:3}, we implement an infinite iteration of normal form reductions
and transform  the equation \eqref{NLS4}
into the normal form equation \eqref{NFE1}
for regular solutions.
The following proposition summarizes the properties of the multilinear operators in \eqref{NFE1}.
Given $R>0$, 
we use $B_R$ to denote 
the ball of radius $R$ 
centered at the origin
in various function spaces.

\begin{proposition}
\label{PROP:bound}
Let $1 \leq p < \infty$
and $T>0$.
Then, there exist
time-dependent  multilinear operators 
$\{\N_0^{(j)}\}_{j=2}^{\infty}$, 
$\{\N_{1}^{(j)}\}_{j=1}^{\infty}$, and 
 $\{\RR^{(j)}\}_{j=1}^{\infty}$, depending on the parameter 
  $K = K(R) \geq 1$
  such that any regular solution 
 $\uu \in C([-T,T]; \F L^{\frac32}(\T))$ 
to   \eqref{NLS4} with $\uu(0) \in B_R\subset \FL^p(\T)$ 
satisfies the following normal form equation:
 \begin{align}
 \begin{split}
\uu(t) - \uu(0) & =     \sum_{j = 2}^\infty  \N^{(j)}_0(\uu)(t)
 - \sum_{j = 2}^\infty \N^{(j)}_0(\uu)(0)  \\
& \hspace{4mm} + \int_0^t \bigg\{
\sum_{j = 1}^\infty \N^{(j)}_1(\uu)(t') + \sum_{j = 1}^\infty \RR^{(j)}(\uu)(t') \bigg\} dt'
\end{split}
\label{NFE2a}
\end{align}

\noi
in $ C([-T,T]; \F L^{\frac 32}(\T))$. 
Moreover, 
 $\{\N_0^{(j)}\}_{j=2}^{\infty}$ are $(2j-1)$-linear operators, 
while $\{\N_1^{(j)}\}_{j=1}^{\infty}$
and $\{\RR^{(j)}\}_{j=1}^{\infty}$
are   $(2j+1)$-linear operators \textup{(}depending on $t\in [-T, T]$\textup{)},  
satisfying the following bounds on $ \F L^p(\T)$:\footnote{Here, we view 
$\N_0^{(j)} = \N_0^{(j)}(t)$, 
$\N_{1}^{(j)} = \N_{1}^{(j)}(t)$, and 
 $\RR^{(j)} = \RR^{(j)}(t)$ as 
 multilinear operators
 acting on $\FL^p(\T)$
 with a parameter $t \in [-T, T]$.
 The same comment applies to $\RR_2^{(j)}$ in \eqref{est-R2}.
}
\begin{align}
\sup_{t \in [-T, T]} \big\| \N_0^{(j)} (t)(f_1, f_2, \cdots, f_{2j-1}) \big\|_{\F L^p(\T)} 
& \le C_{0,j} \prod_{i=1}^{2j-1} \| f_i \|_{\F L^p (\T)},\label{est-N0}\\
\sup_{t \in [-T, T]} \big\| \N_1^{(j)} (t)(f_1, f_2, \cdots, f_{2j+1}) \big\|_{\F L^p(\T)} 
& \le C_{1,j} \prod_{i=1}^{2j+1} \| f_i \|_{\F L^p (\T)}\label{est-N1},\\
\sup_{t \in [-T, T]} \big\| \RR^{(j)}(t) (f_1, f_2, \cdots, f_{2j+1}) \big\|_{\F L^p(\T)} 
& \le C_{0,j} \prod_{i=1}^{2j+1} \| f_i \|_{\F L^p (\T)},\label{est-R1}
\end{align}

\noi
for any $f_i\in \F L^p(\T)$, where 
\begin{align}
C_{0,j}(K)  =C_p \frac{ K^{4(1-j)}}{ j!}
\qquad \text{and}\qquad 
C_{1,j}(K)  = C_p \frac{K^{\frac{16}{p'-1} + 4(1-j)}}{ j!}
\label{decay1}
\end{align}

\noi
for some absolute constant  $C_p>0$  depending only on $p$.
\end{proposition}

In Proposition \ref{PROP:bound}, 
we imposed a strong regularity assumption:
 $\uu \in C([-T,T]; \F L^{\frac32}(\T))$.
 This regularity assumption can be easily relaxed.

\begin{corollary}
\label{COR:1}
Let $1 \leq p < \infty$ and $T>0$.
Suppose that a solution $\uu  \in C([-T, T]; \FL^p(\T))$ to \eqref{NLS4}
admits a sequence of smooth approximating solutions $\{\uu_m\}_{m\in \NB}$
in the sense that 
\textup{(i)} $\uu_m$ is a smooth solution to \eqref{NLS4}
and \textup{(ii)} $\uu_m$ converges to $\uu$ in $ C([-T, T]; \FL^p(\T))$.
Then, $\uu$  satisfies the normal form equation \eqref{NFE2}
in $ C([-T,T]; \F L^{p}(\T))$.

\end{corollary}

 In view of the estimates \eqref{est-N0}, \eqref{est-N1}, and \eqref{est-R1}, 
we see that the right-hand side of \eqref{NFE2a} is convergent
for $\uu \in C([-T, T]; \FL^p(\T))$.
See also the proof of Theorem \ref{THM:3} below.
By using the multilinearity of the operators, 
we only need to estimate the difference
such as
$\N_0^{(j)}(\uu) - \N_0^{(j)}(\uu_m)$.
Note that such a difference contains $O(j)$-many terms
since 
$|a^{2j-1} - b^{2j-1}| \lesssim     
\big(\sum_{k = 1}^{2j-1} a^{2j-1-k}b^{k-1} \big) |a - b |$
has $O(j)$ many terms.
This, however,  does not cause any issue
thanks to the fast decay~\eqref{decay1} of the coefficients $C_{0, j}$ and  $C_{1,  j}$.
Since the proof 
of Corollary \ref{COR:1} is straightforward computation
with \eqref{est-N0}, \eqref{est-N1}, and \eqref{est-R1}, 
we omit details.

We postpone the proof of Proposition \ref{PROP:bound}
 to Section \ref{SEC:3}.
 In the remaining part of this section, we present the proofs of
Theorems \ref{THM:1}, \ref{THM:2}, and \ref{THM:3},  assuming Proposition \ref{PROP:bound}.
In Subsection \ref{SUBSEC:NLS}, we discuss the case of the 
(unrenormalized) NLS \eqref{NLS1}.

We first present the proof of Theorem \ref{THM:3}.

\begin{proof}[Proof of Theorem \ref{THM:3}.]

Given  $1 \leq p < \infty$, let $\uu_0 \in \FL^p(\T)$.
With $K = K(\|\uu_0\|_{\FL^p}) \geq 1$ (to be chosen later), 
define the map $\G_{\uu_0}$ by 
\begin{align*}
\G_{\uu_0} (\uu)(t)
&  : =   \uu_0  +  \sum_{j = 2}^\infty  \N^{(j)}_0(\uu)(t)
 - \sum_{j = 2}^\infty \N^{(j)}_0(\uu)(0)  \notag \\
& \hphantom{X}
+ \int_0^t \bigg(
\sum_{j = 1}^\infty \N^{(j)}_1(\uu)(t') 
+ \sum_{j = 1}^\infty \RR^{(j)}(\uu)(t') \bigg) dt',
\end{align*}

\noi
where the multilinear terms on the right-hand side 
(depending on the choice of $K\geq 1$) are as in  Proposition \ref{PROP:bound}. 
Let $T>0$.
Then, 
by  Proposition \ref{PROP:bound}, we have
\begin{align*}
\| \G_{\uu_0}(\uu)\|_{C_T\F L^p} \le & \| \uu_0\|_{\F L^p} + \sum_{j=2}^{\infty} C_{0,j}(K) \Big( \| \uu_0\|_{\F L^p}^{2j-1} + \| \uu \|_{C_T \F L^p}^{2j-1} \Big)\\
& + T   \sum_{j=1}^\infty \big( C_{1,j} (K)+ C_{0,j} (K)\big) \| \uu\|_{C_T \F L^p}^{2j+1} , 
\end{align*}

\noi
where $C_T\FL^p = C([-T, T]; \FL^p(\T))$.

Let   $R =1+  \| \uu_0\|_{\F L^p}$.
Then from \eqref{est-N0}, \eqref{est-N1}, and \eqref{est-R1}, 
we have
 \begin{align}
 \label{contra1}
\begin{split}
\| \G_{\uu_0}(\uu) \|_{C_T\F L^p} 
&  \le   R 
 + C \sum_{j=2}^{\infty} 
\frac{K^{4(1-j)} R^{2j-1} }{j!} 
+ C \sum_{j=2}^{\infty} 
\frac{K^{4(1-j)}(2R)^{2j-2}}{j!}  \|\uu\|_{C_T \FL^p} \\
&\hphantom{X}
+ 
C T \bigg\{ \sum_{j=1}^\infty 
\frac{K^{\frac{16}{p'-1} + 4(1-j)} (2R)^{2j}  }{ j!}
+ \sum_{j=1}^\infty  
 \frac{K^{4(1-j)} (2R)^{2j} }{j!} 
 \bigg\} 
\| \uu\|_{C_T \FL^p}\\
\end{split}
\end{align}

\noi 
for any  $\uu\in B_{2R} \subset C ([-T,T]; \F L^p(\T))$.
The series in \eqref{contra1} are obviously convergent
for any $K \geq 1$ thanks to the fast decay in $j$
but by choosing  $K = K(R, p) \gg 1$ sufficiently large, 
we can guarantee that
\[
  C \sum_{j=2}^{\infty} 
\frac{K^{4(1-j)} R^{2j-1} }{j!} \le \frac1{10}
\qquad \text{and}\qquad 
 C \sum_{j=2}^{\infty} 
\frac{K^{4(1-j)}(2R)^{2j-2}}{j!} \le \frac1{10}.
\] 

\noi
Note that the third series in \eqref{contra1} has 
non-negative powers of $K$  for $1\le j < \frac{4}{p'-1}$,
while a power of K does not appear in the fourth series when $j = 1$.
These terms can be controlled by choosing $T = T(K, R) = T(R)>0$
sufficiently small.
As a result, we obtain
\begin{align*}
\| \G_{\uu_0}(\uu) \|_{C_T\F L^p } \le \frac{11}{10} R + \frac15 \| \uu\|_{C_T \F L^p} < 2R
\end{align*}

\noi 
for any  $\uu\in B_{2R} \subset C ([-T,T]; \F L^p(\T))$.
A similar argument also yields the following difference estimate:
\begin{align}
\label{contra3}
\| \G_{\uu_0}(\uu) - \G_{\uu_0}(\vv)  \|_{C_T\F L^p } \le  \frac15 \| \uu - \vv\|_{C_T \F L^p}.
\end{align}

\noi
In establishing the difference estimate \eqref{contra3}, 
we  need to estimate 
the differences
such as
$\N_0^{(j)}(\uu) - \N_0^{(j)}(\vv)$ 
which contains  $O(j)$-many terms as mentioned above.
This  does not cause any issue
thanks to the fast decay~\eqref{decay1} in $j$ of the coefficients $C_{0, j}$ and $C_{1, j}$.

Therefore, by a standard contraction argument
and a continuity argument,\footnote{The contraction argument
yields uniqueness only in $B_{2R} \subset  C([-T,T]; \FL^p (\T))$
and a continuity argument is needed to extend the uniqueness to the entire
 $C([-T,T]; \FL^p (\T))$. This part of the argument is standard and thus we omit detail.
 See for example \cite{CGKO}.}
we conclude that the normal form equation~\eqref{NFE1}
is unconditionally locally well-posed  in $C([-T,T]; \FL^p (\T))$.
Global well-posedness follows
from the a priori bounds  \eqref{bound1} and~\eqref{non5} on the $\FL^p$-norm of smooth solutions
to \eqref{WNLS1} 
 implying
the same bound for smooth solutions to~\eqref{NLS4} and \eqref{NFE1}.

Lastly, by taking the difference of two solutions
$\uu,\vv \in C([-T,T]; \F L^p(\T))$ with different initial data $\uu_0$ and $\vv_0$,  
we have
\begin{align*}
\| \uu - \vv \|_{C_T \FL^p} \le \frac{11}{10}  \| \uu_0 - \vv_0\|_{\F L^p} + \frac15 \| \uu - \vv \|_{C_T \F L^p} ,
\end{align*}

\noi
which implies the Lipschitz bound \eqref{DN3}
for $T = T(\|\uu_0\|_{\FL^p}, \|\vv_0\|_{\FL^p})>0$ sufficiently small. 
By iterating the Lipschitz bound \eqref{DN3} on short intervals
with the global-in-time bounds \eqref{bound1} and \eqref{non5}, 
we conclude that \eqref{DN3}
for any $T > 0$.
\end{proof}

\subsection{Sensible weak solutions: Proof of Theorem \ref{THM:1}}

In the following, we only show global well-posedness of the renormalized cubic NLS 
\eqref{WNLS1} in the sense of sensible weak solutions
according to Definition \ref{DEF:sol}.
As for  well-posedness in the sense of weak solutions in the extended sense
according to Definition \ref{DEF:sol2}, 
one can simply use Proposition \ref{PROP:bound}
and repeat the argument in \cite{GKO}.

Given $u_0 \in \F L^p(\T)$, let  $\{u_{0, m}\}_{m \in \NB}$ 
be a sequence of smooth functions converging to $u_0$ in $\FL^p(\T)$.
Let  $u_m$ be the smooth solution to \eqref{WNLS1} with $u_m|_{t = 0} = u_m$
and set $\uu_m(t) = S(-t) u_m(t)$.
Then, it follows from Proposition \ref{PROP:bound}
that $\uu_m$ is a solution to the normal form equation \eqref{NFE2a}.
From the Lipschitz bound \eqref{DN3}, we have
\begin{align}
\begin{split}
\| u_m- u_n  \|_{C_T \F L^p}
& = 
\| \uu_m- \uu_n  \|_{C_T \F L^p }\\
&   \leq C(T)  \| \uu_m(0) - \uu_n(0) \|_{\FL^p}
  =  C(T)  \| u_m(0) - u_n(0) \|_{\FL^p}
\end{split}
\label{diff2}
\end{align}

\noi
for all $m,n \ge 1$ and any $T>0$.
This shows that 
$\{ u_m\}_{m \in \NB}$ is a Cauchy sequence 
in $C(\R; \FL^p(\T))$ endowed with the compact-open topology (in time)
and hence converges to some $u_\infty$
 in $C(\R; \FL^p(\T))$.

Now, we prove uniqueness of the limit $u_\infty$, independent of smooth approximating solutions.
Given $u_0 \in \F L^p(\T)$, 
let $\{u_m\}_{m\in \NB}$ and $\{v_n\}_{n\in\NB}$ be 
two sequences of smooth  solutions such that $u_m(0), v_n(0) \to u_0$ in $\FL^p(\T)$
 as $m, n \to \infty$.
Then, by the argument above, 
there exist $u_\infty, v_\infty \in C(\R; \FL^p(\T))$
such that 
$u_m \to u_\infty$ and $v_n \to v_\infty$
in $ C(\R; \FL^p(\T))$
as $m,n \to \infty$. 
Then, by the triangle inequality with  \eqref{DN3} and \eqref{diff2}, we obtain
\begin{align*}
\| u_\infty-v_\infty\|_{C_T \F L^p} 
& \le \| u-u_m\|_{C_T \F L^p} + \| u_m - v_n\|_{C_T \F L^p} + \| v_n-v\|_{C_T \F L^p} \\
& \le \| u-u_m\|_{C_T \F L^p} + C \| u_m(0)-v_n (0 ) \|_{\F L^p} + \| v_n-v\|_{C_T \F L^p}\\
& \longrightarrow 0, 
\end{align*}

\noi
 as $m,n \to  \infty$.
 Therefore, we have $u_\infty = v_\infty$.
 
Lastly, combining this convergence with \eqref{WNLS1}, we obtain
\begin{align}
\N(u_m) - \N(u_n)
= - i \dt (u_m - u_n) - \dx^2 (u_m - u_n)
\longrightarrow 0
\label{diff3}
\end{align}

\noi
in the distributional sense as $m, n \to \infty$.
Therefore, we conclude that \eqref{WNLS1} is globally well-posed
in the sense of sensible weak solutions.

\subsection{Unconditional well-posedness
of the renormalized cubic NLS}

We briefly discuss the proof of Theorem \ref{THM:2}
for the renormalized cubic NLS \eqref{WNLS1}.
Given   $u_0 \in \FL^\frac{3}{2}(\T)$, let $u$ and $v$ be two solutions
to \eqref{WNLS1} with $u |_{t = 0} = v |_{t = 0} = u_0$ in $C([-T, T];\FL^\frac{3}{2}(\T))$ for some $T>0$.
By Proposition \ref{PROP:bound}, 
we see that their interaction representations
$\uu(t) = S(-t) u(t)$ and $\vv(t) = S(-t) v(t)$
satisfy the normal form equation \eqref{NFE1}.
Then, from the unconditional uniqueness
for \eqref{NFE1} in $C([-T, T];\FL^\frac{3}{2}(\T))$ (Theorem \ref{THM:3})
and the unitarity of the linear operator in $\FL^p(\T)$, 
we conclude that $u = v$ in $C([-T, T];\FL^\frac{3}{2}(\T))$.
This proves Theorem \ref{THM:2}.

\subsection{On the cubic NLS} 
\label{SUBSEC:NLS}

We conclude this section by discussing the situation
for the cubic NLS \eqref{NLS1}.
By writing 
\begin{align}
\begin{split}
|u|^2 u & = \bigg(  |{u}|^2 - 2 \int_\T \ |{u}|^2 dx\bigg) u  + 2 \bigg( \int_\T \ |{u}|^2 dx\bigg) u \\
& = \sum_{n_2 \ne n_1, n_3} \ft{{u}}(n_1)\cj{\ft u (n_2)}\ft{{u}}(n_3) 
	e^{i(n_1 - n_2 + n_3)x} - 
	\sum_{n\in \Z} |\ft{{u}}(n)|^2 \ft{{u}}(n) e^{inx}\\
& \hphantom{X}	
+ 2\bigg(\int_\T \ |{u}|^2 dx\bigg) \sum_n \ft{u}(n) e^{inx}\\
& = : \I + \II + \III, 
\end{split}
\label{non2}
\end{align}
\noi
we see that the third term $\III$ is the only difference from the case for the renormalized
cubic NLS \eqref{WNLS1}.
By taking an interaction representation, 
we can write \eqref{NLS1} as
\begin{align}
\dt \ft \uu_n 
& =   \N_1(\uu)(n) +   \RR(\uu)(n)  + \RR_2(\uu)(n), 
 \label{NLS7}
 \end{align}

\noi
where $ \RR_2(\uu)(n)$ is given by  
\begin{align*}
 \RR_2(\uu)(n) 
 = 
 2i \bigg(\int_\T \ |\uu|^2 dx\bigg) \ft{\uu}(n).
\end{align*}

\noi
As compared to \eqref{NLS4}, $\RR_2(\uu)$ is the only difference.
Note that this extra term $\RR_2(\uu)$ imposes the restriction $ p \leq 2$.
As in the case of the renormalized NLS \eqref{WNLS1}, 
we prove the following proposition in Section \ref{SEC:3}.

\begin{proposition}
\label{PROP:bound2}
Let $1 \leq p \leq 2$
and $T>0$.
Then, there exist
time-dependent multilinear operators 
$\{\N_0^{(j)}\}_{j=2}^{\infty}$, 
$\{\N_{1}^{(j)}\}_{j=1}^{\infty}$, 
 $\{\RR^{(j)}\}_{j=1}^{\infty}$, and 
  $\{\RR_2^{(j)}\}_{j=1}^{\infty}$
 depending on the parameter 
  $K = K(R) \geq 1$
  such that the interaction representation $\uu(t) = S(-t) u(t)$
  of any regular solution 
 $u \in C([-T,T]; \F L^{\frac32}(\T))$ 
to   \eqref{NLS1} with $u(0) \in B_R\subset \FL^p(\T)$ 
satisfies the following normal form equation:
 \begin{align}
 \begin{split}
\uu(t) - \uu(0) =  &   \sum_{j = 2}^\infty  \N^{(j)}_0(\uu)(t) 
 - \sum_{j = 2}^\infty \N^{(j)}_0(\uu)(0)  \\
& + \int_0^t \bigg\{
\sum_{j = 1}^\infty \N^{(j)}_1(\uu)(t') + \sum_{j = 1}^\infty \RR^{(j)}(\uu)(t')
+ \sum_{j = 1}^\infty \RR_2^{(j)}(\uu)(t') \bigg\} dt'
\end{split}
\label{NFE4}
\end{align}

\noi
in  $ C([-T,T]; \F L^{\frac32}(\T))$. 
Here, $\{\N_0^{(j)}\}_{j=2}^{\infty}$, 
$\{\N_{1}^{(j)}\}_{j=1}^{\infty}$, 
and  $\{\RR^{(j)}\}_{j=1}^{\infty}$ are as in Proposition~\ref{PROP:bound}, satisfying
the bounds \eqref{est-N0}, \eqref{est-N1}, and \eqref{est-R1}, 
while
  $\{\RR_2^{(j)}\}_{j=1}^{\infty}$
 are $(2j+1)$-linear operators \textup{(}depending on $t\in [-T, T]$\textup{)},  
satisfying the following bound:
\begin{align}
\sup_{t\in [-T,T]} \big\| \RR_2^{(j)} (t) (f_1, f_2, \cdots, f_{2j+1}) \big\|_{\F L^p(\T)} 
& \le C_{0,j} \prod_{i=1}^{2j+1} \| f_i \|_{\F L^p (\T)},\label{est-R2}
\end{align}

\noi
for any $f_i\in \F L^p(\T)$, where 
 $C_{0, j} = C_{0, j}(K) > 0$ is as in \eqref{decay1}.
\end{proposition}

With Proposition \ref{PROP:bound2}, 
we can proceed as in the proof of Theorem \ref{THM:3}
and  prove the following unconditional well-posedness of the normal form equation
\eqref{NFE4} for the cubic NLS \eqref{NLS1}.

\begin{theorem}\label{THM:4}
Let $1\leq p \leq 2$.
Then, the normal form equation \eqref{NFE4}
is unconditionally globally well-posed in $\FL^p(\T)$.
\end{theorem}

Then, 
Theorem \ref{THM:1} for $1\leq p \leq 2$
and Theorem \ref{THM:2} for the cubic NLS \eqref{NLS1}
follow from arguments analogous to those presented above.
We omit details.

\section{Normal form reduction: Proof of Proposition \ref{PROP:bound}} 
\label{SEC:3}

In this section, we implement an infinite iteration of normal form reductions
in the Fourier-Lebesgue space $\FL^p(\T)$, $1\leq p < \infty$,
and prove Proposition \ref{PROP:bound}.
The argument is presented in an inductive manner.
More precisely, we start with  the formulation \eqref{NLS4}
and refer to this case as the first step ($J= 1$).
Define
\begin{align}
\N_1 (\uu) : = \sum_{n\in \Z} \N_1(\uu) (n) e^{inx}
\qquad \text{and}\qquad 
\RR (\uu) : = \sum_{n\in \Z}  \RR(\uu) (n) e^{inx},
\label{Def:N}
\end{align}
where $\N_1(\uu) (n)$ and $\RR(\uu) (n)$ are as in \eqref{NLS4}.
In what follows, we view $\N_{1}$ and $\RR$ as trilinear  operators.

For notational convenience, we set $\RR^{(1)} : = \RR$ and $\N^{(1)} : =\N_1$. 
While we keep the resonant part  $\RR^{(1)}$ as it is,
we  divide the non-resonant part $\N^{(1)}$ into  a ``good" part $\N_{1}^{(1)}$
(nearly resonant part) 
and a ``bad" part $\N_2^{(1)}$ (highly non-resonant part),
depending on  the size of the phase function $\Phi (\bar n)$.
On the one hand, 
the restriction on  the phase function $\Phi (\bar n)$
allows us to establish an effective estimate on  the good part $\N_{1}^{(1)}$.
On the other hand, 
the bad part does not allow for any good estimate.
To exploit fast time oscillation, we then apply a normal form reduction to the bad part  $\N_2^{(1)}$
and turn it into the terms
 $\N_0^{(2)}$, $\RR^{(2)}$, and $\N^{(2)}$
 in the second generation ($J = 2$).
We can easily estimate the terms  $\N_0^{(2)}$ and $\RR^{(2)}$.
As in the first step, we  divide $\N^{(2)}$ into 
a good part $\N^{(2)}_1$  and a bad part $\N^{(2)}_2$, 
 where the threshold is now given by the phase function for the quintilinear term $\N^{(2)}$.
While the good part $\N_{1}^{(2)}$ allows for an effective quintilinear estimate, 
we apply a normal form reduction to the bad part  $\N_2^{(2)}$
and turn it into three  terms
 $\N_0^{(3)}$, $\RR^{(3)}$, and $\N^{(3)}$
 in the third generation $(J = 3)$.
We proceed in an inductive manner.

After applying normal form reductions $J-1$ times, 
we arrive at the three terms
 $\N_0^{(J)}$, $\RR^{(J)}$, and $\N^{(J)}$.
 The main difficulty appears in the last term $\N^{(J)}$.
As in the previous steps, 
we divide $\N^{(J)}$ into a good part $ \N^{(J)}_1$ (with an effective $(2J+1)$-linear estimate) 
and a bad part  $\N^{(J)}_2$.
We then apply a normal form reduction to the bad part  $\N^{(J)}_2$
and iterate this procedure indefinitely.
Under some regularity assumption, 
we show that the error term
  $\N^{(J)}_2$ tends to 0 
  as $J \to \infty$.

In order to carry out the strategy described above, 
we need to address the following four issues:

\begin{itemize}

\item
How do we separate $\N^{(J)}$ into ``good" and ``bad" parts?

\item
How do we estimate these good terms in the $\FL^p(\T)$ when $p\gg 1$?
As we see below, 
 $\N_0^{(J)}$ is $(2J-1)$-linear, while $\RR^{(J)}$ and $\N^{(J)}$
 are $(2J+1)$-linear.

\item
Under what condition, 
does  the remainder term $\N^{(J)}_2$ tends to 0 as $J \to \infty$, 
and if so, in which sense?

\item
We need to show convergence of the series representation
\eqref{NFE2a}. 
\end{itemize}

\noi
We address these issues  in the remaining part of this section.
In the following, we fix $1 \leq p < \infty$.
The major part of this section is devoted to 
studying the renormalized cubic NLS \eqref{WNLS1}.
As for the (unrenormalized) cubic NLS \eqref{NLS1}, 
see Subsection \ref{SUBSEC:NLS2}.

\subsection{Base case: $J = 1$}
\label{SUBSEC:first}

Define the trilinear operators $\N^{(1)}$ and $\RR^{(1)}$
by 
\begin{align}
\begin{split}
\N^{(1)}(\uu_1, \uu_2, \uu_3)
& = 
i \sum_{n\in \Z} e^{inx}\sum_{\substack{n = n_1 - n_2 + n_3\\ n_2\ne n_1, n_3} }
e^{ i \Phi(\bar{n})t } 
\ft \uu_1({n_1}) \cj{\ft \uu_2(n_2)}\ft \uu_3(n_3), \\
\RR^{(1)}(\uu_1, \uu_2, \uu_3)
& = -   i \sum_{n \in \Z} e^{inx} \ft \uu_1(n) \cj{\ft \uu_2(n)} \ft \uu_3 (n),  \\
\end{split}
\label{Def:N2}
 \end{align}

\noi
where $\Phi(\bar n)$ is as in \eqref{Phi}.
For notational simplicity, we set
$\N^{(1)}(\uu) = \N^{(1)}(\uu, \uu, \uu)$, etc. 
when all the three arguments coincide.
Note that this notation is consistent with \eqref{NLS4} and \eqref{Def:N}.
Then, we can write \eqref{NLS4} as
\begin{align}
\dt \uu = \N^{(1)}(\uu)+ \RR^{(1)}(\uu).
\label{XNLS1}
\end{align}

The resonant part satisfies the following trivial estimate.

\begin{lemma}\label{LEM:R1}
Let  $1\le p \le \infty$. Then,  we have
\begin{align}
\| \RR^{(1)}(\uu_1, \uu_2, \uu_3)\|_{\F L^p} 
\leq \prod_{i = 1}^3 \|\uu_i\|_{\F L^p}.
\label{R1_1}
\end{align}
\end{lemma}

\begin{proof}
This is clear from $\ell^p_n \subset \ell^{3p}_n$.
\end{proof}

\begin{remark}\rm
(i) 
In the following, we establish various multilinear estimates.
To simply notations, 
we only state and prove estimates when all arguments agree
with the understanding that they can be easily extended 
to multilinear estimates.
Under this convention,~\eqref{R1_1} is written as
\begin{align*}
\| \RR^{(1)}(\uu)\|_{\F L^p} 
\leq  \|\uu\|_{\F L^p}^3
\end{align*}

\noi
We also use  $\ft \uu_n = \ft \uu_n(t)$ to denote
$\ft{\uu}(n, t)$.
Moreover, given a multilinear operator $\M$, 
we simply use $\M(\uu)(n)$ to denote the Fourier coefficients
of $\M(\uu)$.

\smallskip

\noi
(ii) 
The multilinear operators that appear below
are non-autonomous, i.e.~they depend on a parameter $t \in \R$.
They, however, satisfy estimates uniformly in time
and hence we simply suppress their time dependence.
See \eqref{N11_1} for example.
\end{remark}

\medskip

Next, we consider the non-resonant part $\N^{(1)}$ in \eqref{Def:N2}.
As it is, we can not establish an effective estimate
and hence we divide it into two parts.
Given $K \geq 1$ (to be chosen later)
and $1\leq p < \infty$, 
let $\eps = \eps(p) >0$ be a small positive number  such that 
\begin{align}
p'-1-\eps >0. 
\label{eps0}
\end{align}

\noi
In the following, we simply set 
\begin{align}
\label{eps}
\eps = \frac{p'-1}2 > 0
\end{align}

\noi
such that \eqref{eps0} is satisfied.
Furthermore, we set
\begin{align}
\ta = \frac{4p'}{p'-1- \eps} > 0.
\label{eps2}
\end{align}

We write $\N^{(1)}$ in \eqref{Def:N2} as 
\begin{equation} \label{N1}
\N^{(1)} = \N_{1}^{(1)} + \N_{2}^{(1)},
\end{equation}

\noi
where $\N_{1}^{(1)}$ is the restriction of $\N^{(1)}$
onto $A_1$ (on the Fourier side), where $A_1 = \bigcup_n A_1(n)$ with\footnote{Clearly, the number $3^\ta$ in \eqref{A1} 
does not make any difference at this point.
However, we insert it to match with \eqref{Aj}.
See also \eqref{A2}.}
\begin{align}\label{A1}
A_1(n):= \big\{ (n, n_1, n_2, n_3): & \ n = n_1 - n_2 + n_3, \ n_1, n_3 \ne n, \notag\\
& |\Phi(\bar{n})| = |2(n - n_1) (n - n_3)| \leq (3K)^{\ta}   \big\}
\end{align}

\noi
and $\N_{2}^{(1)} := \N^{(1)} - \N_{1}^{(1)}$. 
Then, the ``good'' part $\N_{1}^{(1)}$ satisfies
the following trilinear estimate.

\begin{lemma}\label{LEM:N11}
Let $\N_{1}^{(1)}$ be as in \eqref{N1}.
Then, we have
\begin{align} \label{N11_1}
\| \N_{1}^{(1)}(\uu)\|_{\F L^p} 
& \lesssim K^{\frac{2\ta}{p'}}  \|\uu\|_{\F L^p}^3,
\end{align}

\noi
where $\ta$ is as in \eqref{eps2}.
\end{lemma}

As in the $p = 2$ case studied in \cite{GKO}, 
the following divisor estimate \cite{HW} plays an important role in the following.
Given an integer $n$, let $d(n)$ denote the number of divisors of $m$.
Then, we have
\begin{equation} \label{divisor}
d(n) \les e^{c\frac{\log n}{\log\log n} }
(= o(n^\delta) \text{ for any }\dl>0).
\end{equation}

\begin{remark}\rm
With  \eqref{eps} and \eqref{eps2}, we have
\[
K^{\frac{2\ta}{p'}}  = K^{\frac{16}{p'-1}}, 
\]

\noi
which appears in  \eqref{decay1}  of Proposition \ref{PROP:bound}.
\end{remark}

\begin{proof}

Fix $n, \mu \in \Z$ with $|\mu| \leq (3K)^{\ta}$.
Then, it follows from the divisor estimate~\eqref{divisor}
that 
there are at most $(3K)^{0+}$ many choices for $n_1$ and $n_3$ (and hence for $n_2$ from $n = n_1 - n_2 + n_3$)
satisfying
\begin{equation}
\label{muu}
\mu = 2(n - n_1) (n - n_3).
\end{equation}

\noi
Hence, we have
\[\sup_n \Bigg( \sum_{|\mu|\leq (3K)^{\ta} }  
\sum_{\substack{n = n_1 - n_2 + n_3\\ n_2\ne n_1, n_3\\\mu = \Phi(\bar{n})} }  1 \Bigg) 
\les \sum_{|\mu|\leq (3K)^{\ta} } (3K)^{0+} \les (3K)^{2\ta} .
\]

\noi
Then, by H\"older's inequality, we have 
\begin{align*}
\|\N_{1}^{(1)} (\uu)\|_{\F L^p} 
& = \Bigg(\sum_n \bigg|
\sum_{|\mu|\leq (3K)^{\ta}  } 
\sum_{\substack{n = n_1 - n_2 + n_3\\ n_2\ne n_1, n_3\\\mu = \Phi(\bar{n})} }
\ft \uu_{n_1} \cj{\ft \uu}_{n_2}\ft \uu_{n_3}
\bigg|^p\Bigg)^\frac{1}{p} \\
& \leq \Bigg\{ \sum_n \bigg( 
\sum_{|\mu|\leq (3K)^{\ta} }  \sum_{\substack{n = n_1 - n_2 + n_3\\ n_2\ne n_1, n_3\\\mu = \Phi(\bar{n})} }  1\bigg)^{\frac{p}{p'}}
\bigg(\sum_{n_1, n_3 \in \Z} |\ft \uu_{n_1}|^p 
|\ft \uu_{n_1 + n_3-n}|^p|\ft \uu_{n_3}|^p\bigg)
\Bigg\}^\frac{1}{p}\\
& \lesssim  K^{\frac{2\ta}{p'} } \|\uu\|_{\F L^p}^3.
\end{align*}

\noi
This proves \eqref{N11_1}.
\end{proof}

\medskip
Now, we apply a normal form reduction
to the remaining highly non-resonant part $\N_{2}^{(1)}$.
More precisely,  we differentiate $\N_{2}^{(1)}$ by parts 
(i.e.~the product rule on differentiation in a reversed order)
and write
\begin{align}
 \N_{2}^{(1)}(\uu) (n)  & =  
   \sum_{ A_1(n)^c} 
\dt\bigg( \frac{e^{ i \Phi(\bar{n})t } }{\Phi(\bar{n})}\bigg)
\ft \uu_{n_1} \cj{\ft \uu}_{n_2}\ft \uu_{n_3} \notag \\
 & = 
  \sum_{ A_1(n)^c} 
 \dt \bigg[
\frac{e^{ i \Phi(\bar{n})t } }{\Phi(\bar{n})}
\ft \uu_{n_1} \cj{\ft \uu}_{n_2}\ft \uu_{n_3}\bigg]  
-  \sum_{A_1(n)^c} 
\frac{e^{ i \Phi(\bar{n})t } }{\Phi(\bar{n})}
\dt\big( \ft \uu_{n_1} \cj{\ft \uu}_{n_2}\ft \uu_{n_3}\big) \notag \\
 & = 
 \dt \bigg[
 \sum_{ A_1(n)^c} 
\frac{e^{ i \Phi(\bar{n})t } }{\Phi(\bar{n})}
\ft \uu_{n_1} \cj{\ft \uu}_{n_2}\ft \uu_{n_3}\bigg]  
-  \sum_{A_1(n)^c} 
\frac{e^{ i \Phi(\bar{n})t } }{\Phi(\bar{n})}
\dt\big( \ft \uu_{n_1} \cj{\ft \uu}_{n_2}\ft \uu_{n_3}\big) \notag \\
& =: \dt \N_{0}^{(2)} (\uu) (n) + \wt \N^{(2)} (\uu) (n). \label{N12}
\end{align}

The boundary term  $\N_0^{(2)}$ can be estimated in a straightforward manner.
Using the equation \eqref{NLS4}, 
we can express  $\wt \N^{(2)}(\uu)(n)$
as a quintilinear form:
\begin{align}
\begin{split}
\wt \N^{(2)} (\uu) (n)
& = -  \sum_{A_1(n)^c} 
\frac{e^{ i \Phi(\bar{n})t } }{\Phi(\bar{n})}
 \Big\{  \RR(\uu)(n_1)
\cj{\ft \uu}_{n_2}\ft \uu_{n_3}
\\
& \hphantom{XXXXXXXXX}
+ 
\ft \uu_{n_1}
\cj{  \RR(\uu)(n_2)}\ft \uu_{n_3}
 + 
\ft \uu_{n_1}
\cj{\ft \uu}_{n_2}   \RR(\uu)(n_3)
\Big\}\\
& = -  \sum_{A_1(n)^c} 
\frac{e^{ i \Phi(\bar{n})t } }{\Phi(\bar{n})}
 \Big\{ \N_1(\uu) (n_1)
\cj{\ft \uu}_{n_2}\ft \uu_{n_3}\\
& \hphantom{XXXXXXXXX}
+ 
\ft \uu_{n_1}
\cj{\N_1(\uu) (n_2)}\ft \uu_{n_3}
 + 
\ft \uu_{n_1}
\cj{\ft \uu}_{n_2}\N_1(\uu)(n_3)
\Big\}\\
& = : \RR^{(2)}(\uu)(n)
+ \N^{(2)}(\uu)(n).
\end{split}
\label{N2}
\end{align}

\noi
In view of \eqref{Def:N2}, 
we regard
$ \RR^{(2)}(\uu)(n)$
and 
$\N^{(2)}(\uu)(n)$
on the right-hand side  as quintilinear forms.
As in the first step, we will need to divide $\N^{(2)}$
into good and bad parts
and apply another normal form reduction to the bad part.
Before proceeding further, 
we first recall the notion of ordered trees introduced in \cite{GKO}.
This allows us to express multilinear terms in a concise manner.

\begin{remark}\rm
\label{exchange}
We formally exchanged the order of the sum and the time differentiation
in the first term
at the third equality.
This can be easily justified in the distributional sense (see Lemma 5.1 in \cite{GKO})
and also in the classical sense if 
$\uu \in C([-T, T]; \FL^\frac{3}{2}(\T))\subset C([-T, T]; L^3(\T))$.
See \cite{GKO}.

\end{remark}

\subsection{Notations: index by trees}
\label{notation}
In this subsection, we recall the notion of ordered trees and relevant  definitions
from \cite{GKO}.

\begin{definition} \label{DEF:tree1} \rm
\textup{(i)} 
Given a partially ordered set $\TT$ with partial order $\leq$, 
we say that $b \in \TT$ 
with $b \leq a$ and $b \ne a$
is a child of $a \in \TT$,
if  $b\leq c \leq a$ implies
either $c = a$ or $c = b$.
If the latter condition holds, we also say that $a$ is the parent of $b$.

\smallskip
\noi
\textup{(ii)}
A tree $\TT$ is a finite partially ordered set satisfying
the following properties.
\begin{itemize}
\item Let $a_1, a_2, a_3, a_4 \in \TT$.
If $a_4 \leq a_2 \leq a_1$ and  
$a_4 \leq a_3 \leq a_1$, then we have $a_2\leq a_3$ or $a_3 \leq a_2$,

\item
A node $a\in \TT$ is called terminal, if it has no child.
A non-terminal node $a\in \TT$ is a node 
with  exactly three children denoted by $a_1, a_2$, and $a_3$,

\item There exists a maximal element $r \in \TT$ (called the root node) such that $a \leq r$ for all $a \in \TT$.
We assume that the root node is non-terminal,

\item $\TT$ consists of the disjoint union of $\TT^0$ and $\TT^\infty$,
where $\TT^0$ and $\TT^\infty$
denote  the collections of non-terminal nodes and terminal nodes, respectively.
\end{itemize}

\noi
The number $|\TT|$ of nodes in a tree $\TT$ is $3j+1$ for some $j \in \NB$,
where $|\TT^0| = j$ and $|\TT^\infty| = 2j + 1$.
Let us denote  the collection of trees in the $j$th generation
by $T(j)$:
\begin{equation*}
T(j) := \{ \TT : \TT \text{ is a tree with } |\TT| = 3j+1 \}.
\end{equation*}

\noi
Note that $\TT \in T(j)$ contains  $j$ parental nodes.

\smallskip

\noi
(iii) (ordered tree) We say that a sequence $\{ \TT_j\}_{j = 1}^J$ is a chronicle of $J$ generations, 
if 
\begin{itemize}
\item $\TT_j \in {T}(j)$ for each $j = 1, \dots, J$,
\item  $\TT_{j+1}$ is obtained by changing one of the terminal
nodes in $\TT_j$ into a non-terminal node (with three children), $j = 1, \dots, J - 1$.
\end{itemize}

\noi
Given a chronicle $\{ \TT_j\}_{j = 1}^J$ of $J$ generations,  
we refer to $\TT_J$ as an {\it ordered tree} of the $J$th generation.
We denote the collection of the ordered trees of the $J$th generation
by $\mathfrak{T}(J)$.
Note that the cardinality of $\mathfrak{T}(J)$ is given by 
\begin{equation} \label{cj1}
 |\mathfrak{T}(J)| = 1\cdot3 \cdot 5 \cdot \cdots \cdot (2J-1)
 = (2J-1)!! =: c_J.
 \end{equation}

\end{definition}

The notion of ordered trees comes with associated chronicles;
it encodes not only the shape of a tree
but also how it ``grew''.
This property will be convenient in encoding 
successive applications
of the product rule for differentiation.
In the following, we simply refer to an ordered tree $\TT_J$ of the $J$th generation
but 
it is understood that there is an underlying chronicle $\{ \TT_j\}_{j = 1}^J$.

Given a tree $\TT$, we associate each terminal node 
$a\in \TT^\infty$ with the Fourier coefficient (or its complex conjugate) of the interaction representation $\uu$ and sum over all possible frequency assignments. 
In order to do this, we introduce the index function $\n$ assigning frequencies 
to all the nodes in $\TT$ in a consistent manner.

\begin{definition}[index function] \label{DEF:tree4} \rm
Given an ordered tree $\TT$ (of the $J$th generation for some $J \in \NB$), 
we define an index function ${\bf n}: \TT \to \Z$ such that,
\begin{itemize}
\item[(i)] $n_a = n_{a_1} - n_{a_2} + n_{a_3}$ for $a \in \TT^0$,
where $a_1, a_2$, and $a_3$ denote the children of $a$,
\item[(ii)] $\{n_a, n_{a_2}\} \cap \{n_{a_1}, n_{a_3}\} = \emptyset$ for $a \in \TT^0$,

\item[(iii)] $|\mu_1| := |2(n_r - n_{r_1})(n_r - n_{r_3})| >(3K)^\ta$,\footnote{Recall that we are on $A_1(n)^c$. See \eqref{A1}.}
 where $r$ is the root node,

\end{itemize}

\noi
where  we identified ${\bf n}: \TT \to \Z$ 
with $\{n_a \}_{a\in \TT} \in \Z^\TT$. 
We use 
$\mathfrak{N}(\TT) \subset \Z^\TT$ to denote the collection of such index functions ${\bf n}$.

\end{definition}

\begin{remark} \label{REM:terminal}
\rm Note that ${\bf n} = \{n_a\}_{a\in\TT}$ is completely determined
once we specify the values $n_a$ for $a \in \TT^\infty$.
\end{remark}

\medskip

Given an ordered tree 
$\TT_J$ of the $J$th generation with the chronicle $\{ \TT_j\}_{j = 1}^J$ 
and associated index functions ${\bf n} \in \mathfrak{N}(\TT_J)$,
 we use superscripts to 
  denote  ``generations'' of frequencies.

Fix ${\bf n} \in \mathfrak{N}(\TT_J)$.
Consider $\TT_1$ of the first generation.
Its nodes consist of the root node $r$
and its children $r_1, r_2, $ and $r_3$. 
We define the first generation of frequencies by
\[\big(n^{(1)}, n^{(1)}_1, n^{(1)}_2, n^{(1)}_3\big) :=(n_r, n_{r_1}, n_{r_2}, n_{r_3}).\]

\noi
 The ordered tree $\TT_2$ of the second generation is obtained from $\TT_1$ by
changing one of its terminal nodes $a = r_k \in \TT^\infty_1$ for some $k \in \{1, 2, 3\}$
into a non-terminal node.
Then, we define
the second generation of frequencies by
\[\big(n^{(2)}, n^{(2)}_1, n^{(2)}_2, n^{(2)}_3\big) :=(n_a, n_{a_1}, n_{a_2}, n_{a_3}).\]

\noi
Note that  we have $n^{(2)} = n_k^{(1)} = n_{r_k}$ for some $k \in \{1, 2, 3\}$.
As we see later, this corresponds to introducing a new set of frequencies
after the first differentiation by parts.

After  $j - 1$ steps, the ordered tree $\TT_j$ 
of the $j$th generation is obtained from $\TT_{j-1}$ by
changing one of its terminal nodes $a  \in \TT^\infty_{j-1}$
into a non-terminal node.
Then, we define
the $j$th generation of frequencies by
\[\big(n^{(j)}, n^{(j)}_1, n^{(j)}_2, n^{(j)}_3\big) :=(n_a, n_{a_1}, n_{a_2}, n_{a_3}).\]

\noi
Note that these frequencies
satisfies (i) and (ii) 
in Definition \ref{DEF:tree4}.

Lastly, we use $\mu_j$  to denote the corresponding phase factor introduced at the $j$th generation.
Namely, we have
\begin{align}
\mu_j & = \mu_j \big(n^{(j)}, n^{(j)}_1, n^{(j)}_2, n^{(j)}_3\big)
:= \big(n^{(j)}\big)^2 - \big(n_1^{(j)}\big)^2 + \big(n_2^{(j)}\big)^2- \big(n_3^{(j)}\big)^2 \notag \\
& = 2\big(n_2^{(j)} - n_1^{(j)}\big) \big(n_2^{(j)} - n_3^{(j)}\big)
= 2\big(n^{(j)} - n_1^{(j)}\big) \big(n^{(j)} - n_3^{(j)}\big), \label{mu}
\end{align}

\noi
where the last two equalities hold thanks to (i) in Definition \ref{DEF:tree4}.

\begin{remark}\rm
For simplicity of notation, 
we may drop the minus signs, the complex number $i$, 
and the complex conjugate sign in the following
when they do not play an important role.
\end{remark}


\subsection{Second generation: $J=2$}
\label{SUBSEC:second}

With the ordered tree notion introduced in the previous subsection, we now rewrite \eqref{N2} as
\begin{align}
\label{N2-R}
\wt \N^{(2)} (\uu) (n)
 & = 
 \sum_{\TT_{1} \in \mathfrak{T}(1)}
\sum_{b \in\TT^\infty_1} 
\sum_{\substack{{\bf n} \in \mathfrak{N}(\TT_1)\\{\bf n}_r = n} }
\ind_{A_1(n)^c}\frac{e^{ i {\mu}_1 t } }{{\mu}_1}
\, \RR^{(1)} (\uu) (n_b)
\prod_{a \in \TT^\infty_1 \setminus \{b\}} \ft \uu_{n_{a}}  \notag \\
& \hphantom{X} 
+
\sum_{\TT_{2} \in \mathfrak{T}(2)}\sum_{\substack{{\bf n} \in \mathfrak{N}(\TT_{2})\\{\bf n}_r = n}} 
\ind_{A_1(n)^c}
\frac{e^{ i (\mu_{1} + \mu_2)t } }{{\mu}_1}
\,\prod_{a \in \TT^\infty_{2}} \ft \uu_{n_{a}} \notag\\
& =: \RR^{(2)} (\uu) (n) +  \N^{(2)} (\uu)(n).
\end{align}

\noi
In the first equality, 
we used \eqref{NLS4} and replace $\dt \ft \uu_{n_b}$ by $\RR^{(1)} (\uu) (n_b)$ and $\N^{(1)} (\uu) (n_b)$.
Strictly speaking, the new phase factor may be $\mu_1 - \mu_2$ when the time derivative falls on the complex conjugate.
However, for our analysis, it makes no difference and hence we simply write it as $\mu_1 + \mu_2$.
We apply the same convention for subsequent steps.

Putting \eqref{N12} and \eqref{N2-R} together, we have
\[
\N^{(1)}_2(\uu)(n) = \dt \N_{0}^{(2)} (\uu) (n) + \RR^{(2)} (\uu) (n) +  \N^{(2)} (\uu)(n).
\]

\noi
The boundary term $\N_0^{(2)} (\uu)$ and the ``resonant'' term $\RR^{(2)}$ 
can be bounded in a straightforward manner.

\begin{lemma}
\label{LEM:N_0^2R^2}
Let $1 \leq p < \infty$. Then, we have
\begin{align*}
\| \N_0^{(2)} (\uu) \|_{\F L^p} & \les K^{-4} \| \uu \|_{\F L^p}^3, \\
\| \RR^{(2)} (\uu) \|_{\F L^p} & \les K^{-4} \| \uu \|_{\F L^p}^5.
\end{align*}
\end{lemma} 

\noi
For the proof of Lemma \ref{LEM:N_0^2R^2},
see Lemma \ref{LEM:N^J_0} and \ref{LEM:R^J} with $J=2$.

With $\ta > 0$ as in \eqref{eps2}, 
 we decompose the frequency space\footnote{If we fix $\TT_2 \in  \mathfrak{T}(2)$, 
 then 
  the frequency space of $\N^{(2)}$
 for this fixed $\TT_2$ in \eqref{N2-R} 
 is given by 
 \[\{(n_a, a \in \TT^\infty_2): \n = \{n_a\}_{a\in \TT_2} \in \mathfrak{N}(\TT_2)\}.\]

 \noi
 In view of Remark \ref{REM:terminal}, 
 we can then  identify
 the frequency space of $\N^{(2)}$ for this fixed $\TT_2$ with  $\mathfrak{N}(\TT_{2})$.}
 of $\N^{(2)}$ 
for  fixed $\TT_2 \in  \mathfrak{T}(2)$
 into 
\begin{align}
\label{A2}
A_2 : = \big\{ {\bf n} \in \mathfrak{N}(\TT_{2}): |\mu_1 + \mu_2| \le (5K)^\ta\big\},
\end{align}
and its complement $A_2^c$.
Then we decompose $\N^{(2)}$ as 
\begin{align}
\label{N21}
\N^{(2)} = \N^{(2)}_1 + \N^{(2)}_2, 
\end{align}
where $\N^{(2)}_1 : = \N^{(2)} |_{A_2}$ is defined as the restriction of $\N^{(2)}$ on $A_2$
and $\N^{(2)}_2 : = \N^{(2)} - \N^{(2)}_1$.
Thanks to the restriction \eqref{A2} on the frequencies, we can estimate the first term $\N^{(2)}_1$.

\begin{lemma}
\label{LEM:N_1^2}
Let $1 \leq p < \infty$. Then, we have
\begin{align*}
\| \N_1^{(2)} (\uu) \|_{\F L^p} & \les K^{\frac{2\ta}{p'} -4} \| \uu \|_{\F L^p}^5.
\end{align*}
\end{lemma} 

\noi
For the proof of Lemma \ref{LEM:N_1^2}, see Lemma \ref{LEM:N^J_1} with $J=2$.

As we do not have  a good control on the operator $\N^{(2)}_2$, 
we apply another normal form reduction to $\N^{(2)}_2$.
On the support of $\N^{(2)}_2$, we have
\begin{align}
\label{A1A2}
|\mu_1| > (3K)^\ta \qquad \text{and} \qquad 
 |\mu_1 + \mu_2| > (5K)^\ta .
\end{align}

\noi
By applying differentiation by parts once again, we have
\begin{align}
\label{N22}
\N^{(2)}_2 (\uu) (n) 
& = \dt \bigg[ 
\sum_{\TT_2 \in \mathfrak{T}(2)}\sum_{\substack{{\bf n} \in \mathfrak{N}(\TT_2)\\{\bf n}_r = n}} 
\ind_{\bigcap_{j = 1}^2 A_j^c}
\frac{e^{ i (\mu_1 + \mu_2) t } }{\mu_1(\mu_1 + \mu_2) }
\, \prod_{a \in \TT^\infty_2} \ft \uu_{n_{a}}
\bigg]\notag \\
& \hphantom{X}  +
\sum_{\TT_2 \in \mathfrak{T}(2)}\sum_{\substack{{\bf n} \in \mathfrak{N}(\TT_2)\\{\bf n}_r = n}} 
\ind_{\bigcap_{j = 1}^2 A_j^c}
\frac{e^{ i (\mu_1 + \mu_2) t } }{\mu_1(\mu_1 + \mu_2) }
\, \dt\bigg( \prod_{a \in \TT^\infty_2} \ft \uu_{n_{a}} \bigg)
 \notag \\
& = \dt \bigg[ 
\sum_{\TT_2 \in \mathfrak{T}(2)}\sum_{\substack{{\bf n} \in \mathfrak{N}(\TT_2)\\{\bf n}_r = n}} 
\ind_{\bigcap_{j = 1}^2 A_j^c}
\frac{e^{ i (\mu_1 + \mu_2) t } }{\mu_1(\mu_1 + \mu_2) }
\, \prod_{a \in \TT^\infty_2} \ft  \uu_{n_{a}}
\bigg]\notag \\
& \hphantom{X} 
+
 \sum_{\TT_{2} \in \mathfrak{T}(2)}
\sum_{b \in\TT^\infty_2} 
\sum_{\substack{{\bf n} \in \mathfrak{N}(\TT_2)\\{\bf n}_r = n} }
\ind_{\bigcap_{j = 1}^2 A_j^c}
\frac{e^{ i (\mu_1 + \mu_2)  t } }{{\mu}_1(\mu_1 + \mu_2) }
\, \RR^{(1)} (\uu) ({n_b})
\prod_{a \in \TT^\infty_{J-1} \setminus \{b\}}\ft  \uu_{n_{a}}  \notag \\
& \hphantom{X} 
+
\sum_{\TT_{3} \in \mathfrak{T}(3)}\sum_{\substack{{\bf n} \in \mathfrak{N}(\TT_{3})\\{\bf n}_r = n}} 
\ind_{\bigcap_{j = 1}^2 A_j^c}
\frac{e^{ i (\mu_1 + \mu_2 + \mu_3) t } }{{\mu}_1(\mu_1 + \mu_2)}
\,\prod_{a \in \TT^\infty_{3}}\ft  \uu_{n_{a}} \notag\\
& =: \dt \N^{(3)}_0 (\uu) (n) + \RR^{(3)} (\uu) (n) + \N^{(3)} (\uu),
\end{align}
where the summations are restricted to  \eqref{A1A2}.
As for the last term $\N^{(3)} (\uu)$, we need to decompose it 
into $\N^{(3)}_1 (\uu)$ and $\N^{(3)}_2 (\uu)$,
according to the further restriction
\begin{align}
\label{A3}
A_3 : = \big\{ {\bf n} \in \mathfrak{N}(\TT_{3}): |\mu_1 + \mu_2 + \mu_3| \le (7K)^\ta\big\}.
\end{align}

\noi
On the one hand, the modulation restrictions \eqref{A1}, \eqref{A2}, and \eqref{A3} allow us to estimate  
operators $\N^{(3)}_0$, $\RR^{(3)}$, and $\N^{(3)}_1$;
see   Lemmas \ref{LEM:N_0^2R^2} and \ref{LEM:N_1^2} below.
On the other hand, we apply another normal form reduction to $\N^{(3)}_2$.
In this way, we iterate  normal form reductions in an indefinite manner.

\subsection{General step: $J$th generation}
\label{SUBSEC:J}

In this subsection, 
we consider the general $J$th step of normal form reductions.
Before doing so, let us first go over the first two steps
studied in Subsections \ref{SUBSEC:first} and \ref{SUBSEC:second}.
Write  \eqref{XNLS1}  as
\[
\dt \uu =  \RR^{(1)} (\uu) + \N^{(1)}_1(\uu) + \N^{(1)}_2(\uu).
\]

\noi
The first two terms on the right-hand side admit good estimates; see Lemmas \ref{LEM:R1} and \ref{LEM:N11}.
We then applied the first step of normal form reductions
to the troublesome term $\N^{(1)}_2(\uu)$ and obtained
\begin{align*}
\dt \uu & = 
 \dt \N^{(2)}_0 (\uu) + \sum_{j = 1}^2 \RR^{(j)} (\uu) + \sum_{j = 1}^2 \N^{(j)}_1(\uu) + \N^{(2)}_2(\uu).
\end{align*}

\noi
See \eqref{N12}, \eqref{N2-R}, and \eqref{N21}.
Note that only the last term $\N^{(2)}_2(\uu)$ can not be estimated in a direct manner.
By applying a normal form reduction once again, 
we obtained 
\begin{align}
\dt \uu & = 
 \sum_{j = 2}^3 \dt \N^{(j)}_0 (\uu) + \sum_{j = 1}^3 \RR^{(j)} (\uu) + \sum_{j = 1}^3 \N^{(j)}_1(\uu) + \N^{(3)}_2(\uu).
\label{X3}
\end{align}

\noi
See \eqref{N22}.
Once again, all the terms in \eqref{X3}, except for the last term 
 $\N^{(3)}_2(\uu)$,  admit good estimates; see Lemmas \ref{LEM:N^J_0}, \ref{LEM:R^J}, and \ref{LEM:N^J_1} below. 
We then apply the third step of normal form reductions 
to  $\N^{(3)}_2(\uu)$.
We can formally iterate this process. 
In particularly, after applying normal form reductions $J-1$ times, 
we would arrive at 
\begin{align}
\dt \uu & = 
 \sum_{j = 2}^{J} \dt \N^{(j)}_0 (\uu) + \sum_{j = 1}^{J} \RR^{(j)} (\uu) + \sum_{j = 1}^{J} \N^{(j)}_1(\uu) 
 + \N^{(J)}_2(\uu).
\label{XJ-1}
\end{align}

In the following, we define each term on the right-hand side of \eqref{XJ-1} properly.
With $\mu_j$ as in \eqref{mu}, 
define
$\wt{\mu}_{j}$  by 
\begin{align*}
 \wt{\mu}_{j} := \sum_{k = 1}^{j} \mu_k.
\end{align*}

\noi
We then set
\begin{equation} \label{Aj}
A_{j} := \big\{ |\wt{\mu}_{j}| \le ((2j+1)K)^{\ta} \big\}, 
\end{equation}

\noi
where $\theta > 0$ is as in \eqref{eps2}.
Given $j \in \NB$, 
we define 
$\N^{(j)}_2 (\uu) (n)$
by 
\begin{align}
\N^{(j)}_2 (\uu) (n)
 =
\sum_{\TT_{j} \in \mathfrak{T}(j)}\sum_{\substack{{\bf n} \in \mathfrak{N}(\TT_{j})\\{\bf n}_r = n}} 
\ind_{\bigcap_{k = 1}^{j} A_k^c}
\frac{e^{ i \wt{\mu}_{j}t } }{\prod_{k = 1}^{j-1} \wt{\mu}_k}
\,\prod_{a \in \TT^\infty_{j}}  \ft \uu_{n_{a}}.
\label{N2j}
\end{align}

\noi
Note that this definition is consistent with 
$\N_2^{(1)}$, $\N_2^{(2)}$, and $\N_2^{(3)}$ that we saw in the previous subsections.
By applying a normal form reduction to \eqref{N2j}
with \eqref{XNLS1}, 
we obtain
\begin{align} 
\N^{(j)}_2 (\uu) (n)
& = \dt \bigg[
\sum_{\TT_j \in \mathfrak{T}(j)}\sum_{\substack{{\bf n} \in \mathfrak{N}(\TT_{j})\\{\bf n}_r = n}} 
\ind_{\bigcap_{k = 1}^{j} A_k^c}
\frac{e^{ i \wt{\mu}_{j}t } }{\prod_{k = 1}^{j} \wt{\mu}_k}
\, \prod_{a \in \TT^\infty_{j}} \ft \uu_{n_{a}}
\bigg]\notag \\
& \hphantom{X} 
+
 \sum_{\TT_{j} \in \mathfrak{T}(j)}
\sum_{\substack{{\bf n} \in \mathfrak{N}(\TT_{j})\\{\bf n}_r = n} }
\sum_{b \in\TT^\infty_{j}} 
\ind_{\bigcap_{k = 1}^{j} A_k^c}
\frac{e^{ i \wt{\mu}_{j}t } }{\prod_{k = 1}^{j} \wt{\mu}_k}
\, \RR^{(1)}(\uu)({n_b})
\prod_{a \in \TT^\infty_{j} \setminus \{b\}}  \ft \uu_{n_{a}}  \notag \\
& \hphantom{X} 
+
 \sum_{\TT_{j} \in \mathfrak{T}(j)}
\sum_{\substack{{\bf n} \in \mathfrak{N}(\TT_{j})\\{\bf n}_r = n} }
\sum_{b \in\TT^\infty_{j}} 
\ind_{\bigcap_{k = 1}^{j} A_k^c}
\frac{e^{ i \wt{\mu}_{j}t } }{\prod_{k = 1}^{j} \wt{\mu}_k}
\, \N^{(1)}(\uu)({n_b})
\prod_{a \in \TT^\infty_{j} \setminus \{b\}}  \ft \uu_{n_{a}}  \notag \\
& = \dt \bigg[
\sum_{\TT_j \in \mathfrak{T}(j)}\sum_{\substack{{\bf n} \in \mathfrak{N}(\TT_{j})\\{\bf n}_r = n}} 
\ind_{\bigcap_{k = 1}^{j} A_k^c}
\frac{e^{ i \wt{\mu}_{j}t } }{\prod_{k = 1}^{j} \wt{\mu}_k}
\, \prod_{a \in \TT^\infty_{j}} \ft \uu_{n_{a}}
\bigg]\notag \\
& \hphantom{X} 
+
 \sum_{\TT_{j} \in \mathfrak{T}(j)}
\sum_{\substack{{\bf n} \in \mathfrak{N}(\TT_{j})\\{\bf n}_r = n} }
\sum_{b \in\TT^\infty_{j}} 
\ind_{\bigcap_{k = 1}^{j} A_k^c}
\frac{e^{ i \wt{\mu}_{j}t } }{\prod_{k = 1}^{j} \wt{\mu}_k}
\, \RR^{(1)}(\uu)({n_b})
\prod_{a \in \TT^\infty_{j} \setminus \{b\}}  \ft \uu_{n_{a}}  \notag \\
& \hphantom{X} 
+
\sum_{\TT_{j+1} \in \mathfrak{T}(j+1)}\sum_{\substack{{\bf n} \in \mathfrak{N}(\TT_{j+1})\\{\bf n}_r = n}} 
\ind_{\bigcap_{k = 1}^{j} A_k^c}
\frac{e^{ i \wt{\mu}_{j+1}t } }{\prod_{k = 1}^{j} \wt{\mu}_k}
\,\prod_{a \in \TT^\infty_{j+1}}  \ft  \uu_{n_{a}} \notag\\
& =: \dt \N^{(j+1)}_0 (\uu) (n)+ \RR^{(j+1)} (\uu) (n) + \N^{(j+1)} (\uu) (n).
\label{N2jj}
\end{align}

\noi
Here, we formally exchanged the order of the sum and the time differentiation, which can be justified. 
See Remark~\ref{exchange}.
As in Subsections \ref{SUBSEC:first} and \ref{SUBSEC:second}, 
we divide $\N^{(j+1)} $ into 
\begin{equation}
\N^{(j+1)} = \N^{(j+1)}_1  + \N^{(j+1)}_2 ,
\label{P0}
\end{equation}

\noi
where $\N^{(j+1)}_1 (\uu) $ is the restriction of $\N^{(j+1)} (\uu)$
onto $A_{j+1}$
and 
$\N^{(j+1)}_2 (\uu):= \N^{(j+1)} (\uu) - \N^{(j+1)}_1 (\uu) $.
This allows us to define all the terms appearing in \eqref{XJ-1}
in an inductive manner
by applying a normal form reduction to 
$\N^{(j+1)}_2$.

In the remaining part of this subsection, we estimate
the multilinear operators
$\N^{(j)}_0$, $\RR^{(j)}$, and $\N_1^{(j)}$.

\begin{lemma}\label{LEM:N^J_0}
Let $1 \leq p < \infty$.
Then, there exists $C_p > 0$ such that 
\begin{align} 
\| \N^{(j)}_0(\uu)\|_{ \F L^{p}} & \le 
C_p \frac{K^{4(1-j)}}{((2j-1)!!)^2 }
 \|\uu\|_{\F L^{p}}^{2j-1}
  \label{N^J_0-1}
\end{align}

\noi
for any integer $j \geq 2$ and $K \geq 1$.
\end{lemma}

\begin{proof}

From \eqref{N2jj} (with $j+1$ replaced by $j$), 
we have 
\begin{align*}
 \N^{(j)}_0 (\uu) (n)= \sum_{\TT_{j-1} \in \mathfrak{T}(j-1)}
 \sum_{\substack{{\bf n} \in \mathfrak{N}(\TT_{j-1})\\{\bf n}_r = n}} 
\ind_{\bigcap_{k = 1}^{j-1} A_k^c}
\frac{e^{ i \wt{\mu}_{j-1}t } }{\prod_{k = 1}^{j-1} \wt{\mu}_k}
\, \prod_{a \in \TT^\infty_{j-1}} \ft \uu_{n_{a}}.
\end{align*}

\noi
Then, by H\"older's inequality with \eqref{cj1}, we have
\begin{align}
\label{N^J_0-3}
\| \N^{(j)}_0(\uu)\|_{\F L^{p}}
& \le \Bigg\| 
 \sum_{\TT_{j-1} \in \mathfrak{T}(j-1)} 
 \bigg( \sum_{\substack{{\bf n} \in \mathfrak{N}(\TT_{j-1})\\{\bf n}_r = n}} 
\frac{\ind_{\bigcap_{k = 1}^{j-1} A_k^c}} {\prod_{k = 1}^{j-1} |\wt{\mu}_k|^{p'}} \bigg)^{\frac1{p'}}
 \bigg( \sum_{\substack{{\bf n} \in \mathfrak{N}(\TT_{j-1})\\{\bf n}_r = n}} 
  \prod_{a \in \TT^\infty_{j-1}} | \ft \uu_{n_{a}}|^p
 \bigg)^{\frac1{p}}\bigg\|_{\l^p_n}\notag \\
& \le 
\sup_{\substack{\TT_{j-1} \in \mathfrak{T}(j-1)\\n \in \Z}}
\bigg( \sum_{\substack{{\bf n} \in \mathfrak{N}(\TT_{j-1})\\{\bf n}_r = n}} 
\frac{\ind_{\bigcap_{k = 1}^{j-1} A_k^c}} {\prod_{k = 1}^{j-1} |\wt{\mu}_k|^{p'}}
 \bigg)^{\frac{1}{p'}}  \notag\\
& \hphantom{XXX}
\times 
\sum_{\TT_{j-1} \in \mathfrak{T}(j-1)}
\Bigg\|
\bigg(\sum_{\substack{{\bf n} \in \mathfrak{N}(\TT_{j-1})\\{\bf n}_r = n}} 
\prod_{a \in \TT^\infty_{j-1}} | \ft \uu_{n_a}|^p
\bigg)^\frac{1}{p} \Bigg\|_{\l^p_n} \notag \\
& \le  (2j-3)!!
\sup_{\substack{\TT_{j-1} \in \mathfrak{T}(j-1)\\n \in \Z}}
\bigg( \sum_{\substack{{\bf n} \in \mathfrak{N}(\TT_{j-1})\\{\bf n}_r = n}} 
\frac{\ind_{\bigcap_{k = 1}^{j-1} A_k^c}} {\prod_{k = 1}^{j-1} |\wt{\mu}_k|^{p'}}
 \bigg)^{\frac{1}{p'}}  
\|\uu\|_{\F L^p}^{2j-1}. 
\end{align}

\noi
In the last step, 
 we used 
\[
\Bigg( \sum_n  
\sum_{\substack{{\bf n} \in \mathfrak{N}(\TT_{j-1})\\{\bf n}_r = n}} 
\prod_{a \in \TT^\infty_{j-1}} | \ft \uu_{n_a}|^p
\Bigg)^\frac{1}{p} 
=  \|\uu\|_{\F L^p}^{2j-1}.
\]

We claim that 
\begin{align}
\label{BJ1}
\sup_{\substack{\TT_{j-1} \in \mathfrak{T}(j-1)\\n \in \Z}}
\bigg( \sum_{\substack{{\bf n} \in \mathfrak{N}(\TT_{j-1})\\{\bf n}_r = n}} 
\frac{\ind_{\bigcap_{k = 1}^{j-1} A_k^c}} {\prod_{k = 1}^{j-1} |\wt{\mu}_k|^{p'}}
 \bigg)^{\frac{1}{p'}}  
\le B_p^{j-1} K^{4(1-j)} ((2j -1)!! )^{-4},
\end{align}

\noi
where $B_p>0$ is a constant depending only on $p$.
Then, 
by setting
\[
C_p : = \sup_{j \ge 2} \bigg(  \frac{B_p^{j-1}}{(2j-1)!!} \bigg) < \infty, 
\]

\noi
we see that 
\eqref{N^J_0-1} follows from 
\eqref{N^J_0-3} and \eqref{BJ1}.

It remains to prove \eqref{BJ1}.
First, note that
given any small $\eps > 0$, there exists $C = C(\eps) > 0$ such that  
\begin{align}
\label{count}
\sup_{\substack{\TT_{j-1} \in \mathfrak{T}(j-1)\\ n \in \Z}}  
\big\{{\bf n} \in \mathfrak{N}(\TT_{j-1})\, : \, {\bf n}_r = n, |\wt \mu_k| = \al_k, k= 1, \dots, j-1 \big\} 
\le C^{j-1} \prod_{k=1}^{j-1} |\al_k|^{\eps} ,
\end{align}

\noi
See Lemma 8.16 in \cite{OW3}
for an analogous statement.
It follows from 
 the divisor estimate~\eqref{divisor}
 that 
for fixed $n^{(k)}$ and $\mu_k$,
there are at most $O(|\mu_k|^{0+})$ 
many choices for $n^{(k)}_1$, $n^{(k)}_2$, and $n^{(k)}_3$.
Noting that $|\mu_k| \leq |\al_k| +  | \al_{k-1}|$, 
we can iterate
this argument  from $k = 1$ to $ j-1$
and obtain \eqref{count}.

From  \eqref{Aj} and \eqref{count} with \eqref{eps2}, we have
\begin{align*}
\text{LHS of \eqref{BJ1}}
& \le C^{j-1} \prod_{k = 1}^{j-1} 
\bigg( \sum_{\substack{|\wt{\mu}_k|>((2k+1)K)^{\ta}\\ k = 1, \dots, j-1} }
 \frac{1}{|\wt{\mu}_k|^{p' -\eps}} \bigg)^\frac{1}{p'}\\
& \le C^{j-1} \prod_{k=1}^{j-1} 
 \bigg( \int_{((2k+1)K)^{\ta} }^\infty   t^{-p'+ \eps} \,dt \bigg)^\frac{1}{p'}\\
& = B_p^{(j-1)} K^{4(1-j)} ((2j -1)!! )^{-4}.
\end{align*}

\noi
Recalling that  $\eps$  in \eqref{eps} depends only on  $p$, 
we see that $B_p$ and hence $C_p$ depend only on $1\leq p  < \infty$.
This completes the proof of Lemma \ref{LEM:N^J_0}.
\end{proof}

As a consequence of Lemma \ref{LEM:N^J_0}
with Lemma \ref{LEM:R1}, 
we obtain the following estimate on~$\RR^{(j)}$.

\begin{lemma}\label{LEM:R^J}
Let $1 \leq p < \infty$.
Then, there exists $C_p > 0$ such that 
\begin{align} 
\| \RR^{(j)}(\uu)\|_{\F L^p} 
& \le C_p
\frac{ (2j-1)K^{4(1-j)}}{((2j-1)!!)^2 }
\|\uu\|_{\F L^p}^{2j+1} 
\label{R^J-1}
\end{align}

\noi
for any $j \in \NB$ and $K \geq 1$.

\end{lemma}


\begin{proof}
When $j = 1$, this is precisely Lemma \ref{LEM:R1}.
Let  $j \geq 2$. 
Note that $\RR^{(j)}(\uu)$ is nothing but  $\N^{(j)}_0(\uu)$
by replacing $\ft \uu_{n_b}$ with $\RR^{(1)}(\uu)(n_b)$
for $b \in \TT_j^\infty$
and summing over $b \in \TT_j^\infty$.
Then,~\eqref{R^J-1} follows from 
Lemma \ref{LEM:N^J_0} with Lemma \ref{LEM:R1}
and noting that 
 given $\TT_j \in \mathfrak{T}(j-1)$, we have $\#\{b: b \in \TT^\infty_{j-1}\} = 2j-1$.
This extra factor $2j-1$ does not cause a problem thanks to the fast decaying constant 
in \eqref{R^J-1}.
\end{proof}

Lastly, we estimate $\N^{(j)}_1(\uu)$, 
namely, the restriction of $\N^{(j)}$ onto $A_j$.

\begin{lemma}\label{LEM:N^J_1}
Let $1 \le p < \infty$.
Then, there exists $C_p > 0$ such that 
\begin{align} 
\label{N^J-1}
\| \N^{(j)}_1 (\uu)\|_{\F L^{p}} 
& \le  C_p
 \frac{ K^{\frac{2\ta}{p'}+4(1-j)}}{((2j-1)!!)^2 } 
 \|\uu\|_{\F L^{p}}^{2j+1},
\end{align}

\noi
for any $j \in \NB$ and $K \geq 1$.
\end{lemma}

%

\begin{proof}
From \eqref{N2jj} (with $j+1$ replaced by $j$), 
we have 
\begin{align*}
\N_1^{(j)} (\uu) (n)
= \sum_{\TT_{j} \in \mathfrak{T}(j)}\sum_{\substack{{\bf n} \in \mathfrak{N}(\TT_{j})\\{\bf n}_r = n}} 
\ind_{\bigcap_{k = 1}^{j-1} A_k^c \cap A_j}
\frac{e^{ i \wt{\mu}_{j}t } }{\prod_{k = 1}^{j-1} \wt{\mu}_k}
\,\prod_{a \in \TT^\infty_{j}}  \ft  \uu_{n_{a}}.
\end{align*}

\noi
Proceeding as in \eqref{N^J_0-3} with H\"older's inequality, 
we have
\begin{align}
\| \N^{(j)}_1(\uu)\|_{\F L^{p}}
& \le 
\sup_{\substack{\TT_{j} \in \mathfrak{T}(j)\\n \in \Z}}
\bigg( \sum_{\substack{{\bf n} \in \mathfrak{N}(\TT_{j})\\{\bf n}_r = n}} 
\frac{\ind_{\bigcap_{k = 1}^{j-1} A_k^c\cap A_j}} {\prod_{k = 1}^{j-1} |\wt{\mu}_k|^{p'}}
 \bigg)^{\frac{1}{p'}}  \notag\\
& \hphantom{XXX}
\times 
\sum_{\TT_{j} \in \mathfrak{T}(j)}
\Bigg\|
\bigg(\sum_{\substack{{\bf n} \in \mathfrak{N}(\TT_{j})\\{\bf n}_r = n}} 
\prod_{a \in \TT^\infty_{j}} | \ft \uu_{n_a}|^p
\bigg)^\frac{1}{p} \Bigg\|_{\l^p_n} \notag \\
& \le  (2j-1)!!
\sup_{\substack{\TT_{j} \in \mathfrak{T}(j)\\n \in \Z}}
\bigg( \sum_{\substack{{\bf n} \in \mathfrak{N}(\TT_{j})\\{\bf n}_r = n}} 
\frac{\ind_{\bigcap_{k = 1}^{j-1} A_k^c\cap A_j}} {\prod_{k = 1}^{j-1} |\wt{\mu}_k|^{p'}}
 \bigg)^{\frac{1}{p'}}  
\|\uu\|_{\F L^p}^{2j-1}. 
\label{N^J_1-3}
\end{align}

We claim that there exists $B_p>0$ such that 
\begin{align}
\label{BJ2}
 \sup_{\substack{\TT_j \in \mathfrak{T}(j)\\ n \in \Z}}
\Bigg(  \sum_{\substack{{\bf n} \in \mathfrak{N}(\TT_{j})\\{\bf n}_r = n}} 
\frac{\ind_{\bigcap_{k = 1}^{j-1} A_k^c\cap A_j}} {\prod_{k = 1}^{j-1} |\wt{\mu}_k|^{p'}}
 \bigg)^{\frac{1}{p'}}  
\le B_p^{j-1} (2j+1)^{1+\frac{2\ta}{p'}} K^{\frac{2\ta}{p'} + 4(1-j)} ((2j-1)!!)^{-4}.
\end{align}

\noi
Then, the desired estimate \eqref{N^J-1} follows
from 
\eqref{N^J_1-3} and \eqref{BJ2}
by setting
\begin{align*}
C_p := \sup_{j \ge 2} \bigg( \frac{B_p^{j-1}  (2j+1)^{1+ \frac{2\ta}{p'}} }{(2j-1)!!} \bigg).
\end{align*}

It remains to prove \eqref{BJ2}. 
As compared to \eqref{BJ1}
in the proof of Lemma \ref{LEM:N^J_0}, 
the main difference 
is that the summation in \eqref{BJ2} is over ${\bf n} \in \mathfrak{N}(\TT_{j})$ rather than ${\bf n} \in \mathfrak{N}(\TT_{j-1})$. 
Note that 
\begin{align}
\label{count-tree}
 \sum_{\substack{{\bf n} \in \mathfrak{N}(\TT_{j})\\{\bf n}_r = n}} 
=   \sum_{\substack{{\bf n} \in \mathfrak{N}(\TT_{j-1})\\{\bf n}_r = n} }
\sum_{b\in \TT^\infty_{j-1}}
\sum_{\substack{n_b = n_1^{(j)} - n_2^{(j)} + n_3^{(j)} }} .
\end{align}

\noi
With $n_b = n^{(j)}$, 
let $\mu_j$ be as in \eqref{mu}.
Then, thanks to the restriction $A_j$ in \eqref{Aj}, 
we see that 
for fixed $\wt \mu_{j-1}$
there are 
 at most $((2j+1)K)^{\ta}$ many choices of $\mu_{j} $.
 Moreover, we have 
$|\mu_{j}| \le |\wt{\mu}_{j-1}| + ((2j+1)K)^{\ta}$.
Then, by the divisor estimate \eqref{divisor}, 
we conclude that 
\begin{align}
\label{J-sum}
\sum_{b\in \TT^\infty_{j-1}}
\sum_{\substack{n_b= n_1^{(j)} - n_2^{(j)} + n_3^{(j)}  }} \ind_{A_j} 
\les (2j+1) ((2j+1)K)^{\ta} \Big( ((2j+1)K)^{\ta} + |\wt{\mu}_{j-1}| \Big)^{0+}.
\end{align}

\noi
Thus \eqref{BJ2} follows from \eqref{BJ1} together with \eqref{count-tree} and \eqref{J-sum}.
\end{proof}

\subsection{On the error term $\N_2^{(J)}$ and the proof of Proposition \ref{PROP:bound}}

We first prove that the remainder term $\N_2^{(J)} (\uu)$ in \eqref{XJ-1}
tends to zero as $J \to \infty$ under some regularity assumption on~$\uu$.

\begin{lemma}
\label{LEM:error}
Let $\N_2^{(J)}  $ be as in \eqref{N2j} with $j = J$ and $T>0$.
Then, given $\uu \in C([-T, T]; \F L^{\frac32}(\T))$, we have
\begin{align}
\label{error1}
\sup_{t\in [-T,T]}\| \N_2^{(J)}  (\uu) \|_{\F L^\infty} \longrightarrow 0,
\end{align}

\noi
as $J \to \infty$.
\end{lemma}

\begin{proof}
By Young's inequality, we have
\begin{align}
\label{non}
\| \N^{(1)} (\uu) \|_{\F L^\infty} + \| \RR^{(1)} (\uu) \|_{\F L^\infty} 
 \les \|  \uu \|_{\FL^{\frac 32}}^3.
\end{align}

From \eqref{P0} (with $j+1$ replace by $J$), we have
\begin{align}
\N_2^{(J)}(\uu) = \N^{(J)}(\uu) -  \N_1^{(J)}(\uu).
\label{P1}
\end{align}

\noi
Then, by rewriting \eqref{N2jj} (with $j+1$ replace by $J$), we have
\begin{align*}
\N^{(J)} (\uu) (n) & 
= \sum_{\TT_{J} \in \mathfrak{T}({J})}\sum_{\substack{{\bf n} \in \mathfrak{N}(\TT_{J})\\{\bf n}_r = n}} 
\ind_{\bigcap_{k = 1}^{J-1} A_k^c}
\frac{e^{ i \wt{\mu}_{J}t } }{\prod_{k = 1}^{J-1} \wt{\mu}_k}
\, \prod_{a \in \TT^\infty_{J}} \ft \uu_{n_{a}} \notag\\
& = \sum_{\TT_{J-1} \in \mathfrak{T}({J-1})}
\sum_{\substack{{\bf n} \in \mathfrak{N}(\TT_{J-1})\\{\bf n}_r = n}} 
\sum_{b \in \TT^\infty_{J-1}}
\ind_{\bigcap_{k = 1}^{J-1} A_k^c}
\frac{e^{ i \wt{\mu}_{J}t } }{\prod_{k = 1}^{J-1} \wt{\mu}_k} \notag \\
& 
\hphantom{XXXXXXXXXX}
\times (\N^{(1)} + \RR^{(1)})(\uu)(n_b) \prod_{a \in \TT^\infty_{J-1} \setminus \{b\}} \ft \uu_{n_{a}} .
\end{align*}

\noi
Proceeding as in the proof of Lemma \ref{LEM:N^J_0} with \eqref{cj1},  \eqref{BJ1}, 
and \eqref{non}, 
we have
\begin{align}
\label{error4}
\| \N^{(J)}  (\uu) \|_{ \F L^\infty} 
& \les |\TT^\infty_{J-1}| 
\sum_{\TT_{J-1} \in \mathfrak{T}(J-1)}
\sup_{\substack{b \in \TT_{J-1}^\infty \\ n\in \Z }}  \Bigg\{ 
\bigg( \sum_{\substack{{\bf n} \in \mathfrak{N}(\TT_{J-1})\\{\bf n}_r = n}} 
\frac{\ind_{\bigcap_{k = 1}^{J-1} A_k^c}} 
{\prod_{k = 1}^{J-1} |\wt{\mu}_k|^{p'}}
 \bigg)^{\frac{1}{p'}} \notag\\
& \hphantom{X}\times \bigg(\sum_{\substack{{\bf n} \in \mathfrak{N}(\TT_{J-1})\\{\bf n}_r = n}} 
|(\N^{(1)} + \RR^{(1)})(\uu)(n_b)|^p  \prod_{a \in \TT^\infty_{J-1} \setminus \{b\}} |\ft \uu_{n_{a}}|^p 
\bigg)^\frac{1}{p} \Bigg\}  \notag \\
& \le B_p^{J-1}   K^{-4(J-1)} \left((2J-1)!!\right)^{-3}  \|  \uu \|_{\FL^{\frac 32}}^3. \notag\\
& \hphantom{X}
 \times \sup_{\substack{b \in \TT_{J-1}^\infty \\ n\in \Z }}  
 \bigg(\sum_{\substack{{\bf n} \in \mathfrak{N}(\TT_{J-1})\\{\bf n}_r = n}} 
  \prod_{a \in \TT^\infty_{J-1} \setminus \{b\}} |\ft \uu_{n_{a}}|^p 
\bigg)^{\frac1p} \notag\\
&\les B_p^{J-1}  K^{-4(J-1)} \left((2J-1)!!\right)^{-2}  \| \uu \|_{\F L^{\frac32}}^3 \|\uu \|_{ \F L^p}^{2J}
\end{align}

\noi
for any $1 \leq p < \infty$.
Therefore, 
 \eqref{error1} follows from \eqref{P1} with Lemma \ref{LEM:N^J_1}
 and \eqref{error4} with  $ p = \frac32$ by taking $J \to \infty$. 
\end{proof}

We briefly discuss the proof of Proposition \ref{PROP:bound}.

\begin{proof}[Proof of Proposition \ref{PROP:bound}]

In view of Lemmas \ref{LEM:N^J_0}, \ref{LEM:R^J}, and \ref{LEM:N^J_1}, 
it suffices to verify that any solution $\uu \in C([-T,T]; \FL^{\frac32}(\T))$ to \eqref{NLS4} satisfies the normal form equation \eqref{NFE2a}.
By integrating 
\eqref{XJ-1} in time, we have
 \begin{align*}
 \begin{split}
\uu(t) - \uu(0) 
& =     \sum_{j = 2}^{J}  \N^{(j)}_0(\uu)(t) 
 - \sum_{j = 2}^{J} \N^{(j)}_0(\uu) (0)  \\
& 
\hphantom{X}
+ \int_0^t \bigg\{
\sum_{j = 1}^{J} \N^{(j)}_1(\uu)(t') + \sum_{j = 1}^{J} \RR^{(j)}(\uu)(t') \bigg\} dt'
 + \int_0^t \N^{(J)}_2 (\uu) (t') dt'.
\end{split}
\end{align*}

\noi
By letting $J \to \infty$, 
we deduce from Lemma \ref{LEM:error} that 
 the normal form equation \eqref{NFE2a} holds in $C([-T, T]; \FL^\infty(\T))$.

Given $J \geq 2$, 
set
\begin{align*}
X_J 
& = \uu(t)  - \uu(0) -\Bigg[ \sum_{j = 2}^{J}  \N^{(j)}_0(\uu)(t) 
 - \sum_{j = 2}^{J} \N^{(j)}_0(\uu_0)  \\
& \hphantom{XXX}
+ \int_0^t \bigg\{
\sum_{j = 1}^{J} \N^{(j)}_1(\uu)(t') + \sum_{j = 1}^{J} \RR^{(j)}(\uu)(t') \bigg\} dt' \Bigg].
\end{align*}

\noi
On the one hand, it follows from 
Lemmas \ref{LEM:N^J_0}, \ref{LEM:R^J}, and \ref{LEM:N^J_1}
that $X_J$ converges to some $X_\infty$
in $C([-T,T]; \FL^{\frac32}(\T))$ as $J \to \infty$.
See \eqref{contra1}.
On the other hand, we know that $X_J$ converges to 0 in 
$C([-T, T]; \FL^\infty(\T))$.
Therefore, by the uniqueness of the limit, 
we conclude that 
$X_J$ tends to 0 in $C([-T,T]; \FL^{\frac32}(\T))$ as $J \to \infty$.
This shows that the normal form equation \eqref{NFE2a}
holds in  $C([-T,T]; \FL^{\frac32}(\T))$.
\end{proof}

\subsection{On the cubic NLS}
\label{SUBSEC:NLS2}

We conclude this section by briefly discussing the case of
the (unrenormalized) cubic NLS \eqref{NLS1}.
The only difference appears from the extra term $\RR_2$ in~\eqref{NLS7}.
When $j = 1$, we simply set 
$ \RR^{(1)}_2 (\uu) (n) =  \RR_2 (\uu) (n)$.
When we  apply a normal form reduction
and substitute $\dt \uu$ by the equation \eqref{NLS7},
there is an extra term due to $\RR_2$. 
By repeating the computation in \eqref{N2jj}, we have
\begin{align*} 
\N^{(j)}_2 (\uu) (n)
& = \dt \N^{(j+1)}_0 (\uu) (n)+ \RR^{(j+1)} (\uu) (n) + \N^{(j+1)} (\uu) (n)\notag \\
& \hphantom{X} +
 \sum_{\TT_{j} \in \mathfrak{T}(j)}
\sum_{\substack{{\bf n} \in \mathfrak{N}(\TT_{j})\\{\bf n}_r = n} }
\sum_{b \in\TT^\infty_{j}} 
\ind_{\bigcap_{k = 1}^{j} A_k^c}
\frac{e^{ i \wt{\mu}_{j}t } }{\prod_{k = 1}^{j} \wt{\mu}_k}
\, \RR_2(\uu)({n_b})
\prod_{a \in \TT^\infty_{j} \setminus \{b\}}  \ft \uu_{n_{a}}\notag\\
& =: \dt \N^{(j+1)}_0 (\uu) (n)+ \RR^{(j+1)} (\uu) (n) + \N^{(j+1)} (\uu) (n)
+  \RR^{(j+1)}_2 (\uu) (n).
\end{align*}

 Proposition \ref{PROP:bound2}
 follows exactly as for Proposition \ref{PROP:bound}
 once we note the following bound
 on $ \RR^{(j)}_2$.

\begin{lemma}\label{LEM:R^J_2}
Let $1 \leq p \leq 2$.
Then, there exists $C_p > 0$ such that 
\begin{align*} 
\| \RR^{(j)}_2(\uu)\|_{\F L^p} 
& \le C_p
\frac{ (2j-1)K^{4(1-j)}}{((2j-1)!!)^2 }
\|\uu\|_{\F L^p}^{2j+1} 
\end{align*}

\noi
for any $j \in \NB$ and $K \geq 1$.

\end{lemma}

\begin{proof}
This lemma follows from Lemma \ref{LEM:N^J_0} as in the proof of Lemma \ref{LEM:R^J}
once we note that 
\[\| \RR_2(\uu)\|_{\F L^p}  \le \|\uu\|_{\F L^p}^3\]
\noi
when $1 \leq p \leq 2$.
\end{proof}

\appendix

\section{On the persistence of regularity in $\FL^p(\T)$, $1\leq p < 2$}\label{SEC:A}

We first recall the basic definitions and properties
of the Fourier restriction norm spaces $X^{s, b}_p(\T\times \R) $ adapted
to the Fourier-Lebesgue spaces.
Let $\S(\T\times \R)$ be the vector space 
of $C^{\infty}$-functions $u:\R^{2}\rightarrow \C$ 
such that
\[ u(x, t)=u(x+1, t) \qquad 
\text{and} \qquad\sup_{(x, t)\in \R^{2}}|t^{\al}\dt^{\be}\dx^{\g}u(x, t)|<\infty\]

\noi
for any $\al,\beta,\g\in \NB\cup\{0\}$.

\begin{definition} \rm
Let $s,b\in \R$, $1\leq p \leq \infty$. We define the space $X^{s,b}_{p}(\T \times \R )$ as the completion of 
$\S(\T\times \R)$ with respect to the norm 
\begin{equation}
\|u\|_{X^{s,b}_{p}(\T \times \R)}=\|\jb{n}^{s} \jb{\tau+n^{2}}^{b}
\ft u(n, \tau) \|_{\l_{n}^{p}L^{p}_{\tau}(\Z\times \R)}. 
\label{Xsb1}
\end{equation}
\end{definition}

For brevity, we simply  denote $X^{s,b}_{p}(\T \times \R)$ by $X^{s,b}_{p}$.
Recall the following characterization of the $X^{s, b}_{p}$-norm
in terms of the interaction representation $\uu(t) = S(-t) u(t)$:
\begin{equation*} 
\|u\|_{X^{s,b}_{p}}=\|\uu\|_{\FL^{s,p}_{x}\FL^{b,p}_{t}},
\end{equation*}

\noi
where the iterated norm is to be understood in the following sense:
 \[\|\uu\|_{\FL^{s,p}_{x}\FL^{b,p}_{t}}
 :=\|\jb{n}^{s} \jb{\tau}^{b}\ft \uu(n, \tau)\|_{\l^{p}_{n}L^{p}_{\tau}}
 =\big\| \|\jb{n}^{s}\ft \uu (n, t)\|_{\FL^{b,p}_{t}}\big\|_{\l^{p}_{n}}.\]
 
 \noi
 Here, $\FL^{s,p}_{x}(\T)$ is as in \eqref{FL1}
 and $\FL^{b,p}_{t}(\R)$ is defined by the norm:
 \[
\| f \|_{\FL^{b, p} (\R)} : = \|\jb{\tau}^b \ft f (\tau)\|_{L^p_\tau(\R)}.
\] 

 \noi
 Note that these spaces are separable when $p<\infty$.

For any $1\leq p < \infty$ and  $s\in \R$, we have 
\begin{equation}\label{conti}
X^{s,b}_{p} \embeds C(\R; \FL^{s,p}(\T)), \quad \text{if }  b>\frac{1}{p'} = 1 - \frac 1p.
\end{equation} 

\noi
This is a consequence of the dominated convergence theorem along with the following embedding relation:
 $\FL^{b, p}_t \embeds \FL^1_t \embeds C_{t}$,
where the second embedding 
is  the Riemann-Lebesgue lemma.

Given an interval $I \subset \R$, 
 we also define the local-in-time version $X^{s,b}_{p}(I)$ of 
the  $X^{s, b}_{p}$-space
as the collection of functions $u$ such that 
\begin{equation}
\|u\|_{X^{s,b}_{p}(I)}:=\inf \big\{\|v\|_{X^{s,b}_{p}} 
\, : \,  v|_{I}=u \big\}
\label{Xsb3}
\end{equation}

\noi
is finite.

Lastly, we recall  the following linear estimates.
See \cite{FOW} for the proof.

\begin{lemma}\label{LEM:lin} 
\textup{(i) (Homogeneous linear estimate).}
Given $1\leq p\leq \infty$  and $s,b\in \R$, we have 
\begin{equation*}
\|S(t)f\|_{X^{s,b}_{p}([0, T])} \les \|f\|_{\FL^{s,p}}
\end{equation*}

\noi
for any $0 < T \leq 1$.

\medskip

\noi
\textup{(ii) (Nonhomogeneous linear estimate).}
Let $s\in \R$, $1\leq p <  \infty$,  and $-\frac{1}{p}<b'\leq 0\leq b \leq 1+b'.$ 
Then, we have 
\begin{equation}
\label{duhamel0}
\bigg\| \int_0^tS(t-t')F(t')dt'\bigg\|_{X^{s,b}_{p}([0, T])}\les T^{1+b'-b}\|F\|_{X^{s,b'}_{p}([0, T])}
\end{equation}

\noi
for any $0< T \leq 1$.

\end{lemma}

The nonhomogeneous linear estimate \eqref{duhamel0} is based on 
 (2.21) in \cite{Grock1}.
While $p > 1$ is assumed in \cite{Grock1}, the estimate also holds true when $p = 1$.

The following trilinear estimate is the key ingredient 
for establishing the persistence of regularity
in $\FL^p(\T)$, $1\leq p < 2$.

\begin{lemma}\label{LEM:tri}
Let $1 \leq p \leq 2$. Then,  there exists small $\eps> 0$ (independent of $p$) such that 
\begin{align}
\big\| |u|^2 u \big\|_{X^{0, - \frac 12 + 2 \eps}_p([0, T])} 
\les \| u\|_{X^{0, \frac 12 + \eps}_2([0, T])}^2 
\| u\|_{X^{0, \frac 12 + \eps}_p([0, T])} 
\label{T0}
\end{align}

\noi
for any $0 < T \leq 1$.
\end{lemma}

\begin{proof}
By a standard argument, it suffices to prove \eqref{T0} without a time restriction:
\begin{align}
\big\| |u|^2 u \big\|_{X^{0, - \frac 12 + 2 \eps}_p} 
\les \| u\|_{X^{0, \frac 12 + \eps}_2}^2 
\| u\|_{X^{0, \frac 12 + \eps}_p}.
\label{T1}
\end{align}

We first estimate the non-resonant contribution from $\I$ in \eqref{non2}.
We follow the argument in \cite{GH}.
Let $\s_0 = \tau + n^2$
and $\s_j = \tau_j + n_j^2$, $j = 1, 2, 3$.
Then, \eqref{T1} follows once we prove
\begin{align*}
\bigg\| \frac{1}{\jb{\s_0}^{\frac{1}{2}-2\eps}}
\sum_{\substack{n = n_1 - n_2 + n_3\\n \ne n_1, n_3}}
\intt_{\tau = \tau_1 - \tau_2 + \tau_3}\prod_{j = 1}^3 \frac{f_j(n_j, \tau_j)}{\jb{\s_j}^{\frac{1}{2}+\eps}}
d\tau_1 d\tau_2\bigg\|_{\l^p_nL^p_\tau}
\les \bigg(\prod_{j = 1}^2  \| f_j \|_{\l^2_nL^2_\tau}\bigg)
 \| f_3 \|_{\l^p_nL^p_\tau}.
\end{align*}

\noi
By Cauchy-Schwarz and Young's inequalities, it suffices to prove
\begin{align}
\bigg\| \frac{1}{\jb{\s_0}^{1-4\eps}}
\sum_{\substack{n = n_1 - n_2 + n_3\\n \ne n_1, n_3}}
\intt_{\tau = \tau_1 - \tau_2 + \tau_3}\prod_{j = 1}^3 \frac{1}{\jb{\s_j}^{1+2\eps}}
d\tau_1 d\tau_2\bigg\|_{\l^\infty_nL^\infty_\tau} < \infty.
\label{T2}
\end{align}

\noi
From \eqref{Phi}, we have
\[ \prod_{j = 0}^4 \frac{1}{\jb{\s_j}^\eps} \les \frac{1}{\jb{(n - n_1)(n - n_3)}^\eps}.\]

\noi
Then, by estimating the convolutions in $\tau_j$ 
(see Lemma 4.2 in \cite{GTV})
and applying  \eqref{Phi}, we have
\begin{align*}
\text{LHS of } \eqref{T2}
& \les 
\bigg\| \frac{1}{\jb{\s_0}^{1-3\eps}}
\\
& \hphantom{X}\times 
\sum_{\substack{n = n_1 - n_2 + n_3\\n \ne n_1, n_3}}
\frac{1}{\jb{n - n_1}^\eps\jb{n - n_3}^\eps}
 \frac{1}{\jb{\tau + n^2 - 2(n - n_1) (n- n_3)}^{1+\eps}}
\bigg\|_{\l^\infty_nL^\infty_\tau} \\
& \les 
\bigg\| \sum_{k \in \Z\setminus \{0\}}
\frac{1}{\jb{k}^\eps}
 \frac{1}{\jb{\tau + n^2 - 2k}^{1+\eps}}
d(k)
\bigg\|_{\l^\infty_nL^\infty_\tau} < \infty,
\end{align*}

\noi
where we used the divisor estimate \eqref{divisor} in the last step.

Next, we estimate the contribution from the resonant parts $\II$ and $\III$ in \eqref{non2}.
By Young's inequality followed by Cauchy-Schwarz inequality, we have
\begin{align*}
\| \II \|_{X^{0, - \frac 12 + 2 \eps}_p} 
& \les \bigg\|
\intt_{\tau = \tau_1 - \tau_2 + \tau_3} \ft u(n, \tau_1)\cj{\ft u(n, \tau_2)} \ft u(n, \tau_3) 
d\tau_1 d \tau_2\bigg\|_{\l^p_n L^p_\tau} \\
& \les \|  \ft u\|_{\l^\infty_n L^1_\tau}^2
\|  \ft u\|_{\l^p_n L^p_\tau}\\
& \les \| u\|_{X^{0, \frac 12 + \eps}_2}^2 
\| u\|_{X^{0, \frac 12 + \eps}_p} .
\end{align*}

\noi
With \eqref{conti}, we have
\begin{align*}
\| \III \|_{X^{0, - \frac 12 + 2 \eps}_p} 
& \les \| u \|_{L^\infty_t L^2_x}^2 
\|  \ft u(n, \tau)\|_{\l^p_n L^p_\tau}\\
& \les \| u\|_{X^{0, \frac 12 + \eps}_2}^2 
\| u\|_{X^{0, \frac 12 + \eps}_p} .
\end{align*}

\noi
This completes the proof of Lemma \ref{LEM:tri}.
\end{proof}

When $p = 2$, Lemmas \ref{LEM:lin} and \ref{LEM:tri} allow us to prove local well-posedness of \eqref{NLS1}
in $L^2(\T)$, where the local existence time is given by 
\begin{align}
T = T(\|u_0\|_{L^2}) \sim (1 + \|u_0\|_{L^2})^{-\ta} > 0
\label{non3}
\end{align}

\noi
for some $\ta > 0$.
For  $1\leq p < 2$, by applying Lemmas \ref{LEM:lin} and \ref{LEM:tri}, 
we can easily prove local well-posedness of \eqref{NLS1}
in $\FL^p(\T)$, where the local existence time $T$ is given as in~\eqref{non3},
namely, it depends only on the $L^2$-norm of initial data $u_0$.
In this case, a contraction argument yields
\begin{align}
\sup_{t \in [0, T]} \| u(t) \|_{\FL^p} \leq C \|u_0\|_{\FL^p}
\label{non4}
\end{align}

\noi
for some absolute constant $C>0$.
Then, by iterating the local argument with
\eqref{non3} and the $L^2$-conservation, 
we conclude from \eqref{non3} and \eqref{non4} that 
\begin{align}
\sup_{t \in [0, \tau]} \| u(t) \|_{\FL^p} \leq C^{ (1 + \|u_0\|_{L^2})^\ta \tau} \|u_0\|_{\FL^p}
\label{non5}
\end{align}

\noi
for any $\tau > 0$.
This proves global well-posedness of \eqref{NLS1} in $\FL^p(\T)$, $ 1 \leq p < 2$,
with the growth bound \eqref{non5}
on the $\FL^p$-norm of solutions.
A similar argument yields
global well-posedness of the renormalized cubic NLS \eqref{WNLS1} in $\FL^p(\T)$, $ 1 \leq p < 2$.

\begin{ackno}\rm
T.\,O.~was supported by the ERC starting grant 
(no.~637995 ``ProbDynDispEq'').
The authors are grateful to the anonymous referee for a helpful comment that has improved the presentation of this paper.
\end{ackno}


\end{document}